\documentclass[pdflatex,sn-mathphys-num]{sn-jnl}

\usepackage{graphicx}%
\usepackage{multirow}%
\usepackage{amsmath,amssymb,amsfonts}%
\usepackage{amsthm}%
\usepackage{mathrsfs}%
\usepackage[title]{appendix}%
\usepackage{textcomp}%
\usepackage{manyfoot}%
\usepackage{booktabs}%
\usepackage{algorithmicx}%
\usepackage{algpseudocode}%
\usepackage{listings}%
\usepackage{mathtools}
\usepackage{physics}
\usepackage{float}
\usepackage{lineno}
\newcommand{\eps}{\varepsilon}
\DeclareMathOperator*{\argmin}{argmin}

\DeclareMathOperator{\zer}{zer}

\DeclareMathOperator{\dist}{dist}
\DeclareMathOperator{\dif}{d\!}

\DeclareMathOperator{\dom}{dom}

\DeclareMathOperator{\gr}{gr}

\DeclareMathOperator{\res}{\mathsf{J}}

\DeclareMathOperator{\range}{ran}

\DeclareMathOperator{\Id}{Id}

\newcommand{\B}{\mathbb{B}}

\renewcommand{\iff}{\Leftrightarrow}

\renewcommand{\emptyset}{\varnothing}
\newcommand{\eqdef}{\triangleq}
\newcommand{\wlim}{\rightharpoonup}

\newcommand{\scrC}{\mathcal{C}}

\newcommand{\scrG}{\mathcal{G}}
\newcommand{\scrH}{\mathcal{H}}

\newcommand{\scrO}{\mathcal{O}}

\newcommand{\scrU}{\mathcal{U}}

\newcommand{\scrY}{\mathcal{Y}}

\newcommand{\setC}{\mathsf{C}}

\newcommand{\0}{\mathbf{0}}

\newcommand{\Rn}{\R^n}

\newcommand{\R}{\mathbb{R}}

\newcommand{\N}{\mathbb{N}}

\DeclareMathOperator{\NC}{\mathsf{N}}

\newcommand{\bC}{{\mathbf{C}}}

\newcommand{\opA}{\mathsf{A}}
\newcommand{\opB}{\mathsf{B}}

\newcommand{\opD}{\mathsf{D}}
\newcommand{\opM}{\mathsf{M}}

\newcommand{\opF}{{\mathsf{F}}}

\DeclarePairedDelimiter{\inner}{\langle}{\rangle}

\newcommand{\close}{\hfill{\footnotesize$\Diamond$}}

\usepackage[ruled,vlined]{algorithm2e}
\usepackage{algpseudocode}

\newtheorem{theorem}{Theorem}
\newtheorem{proposition}[theorem]{Proposition}%
\newtheorem{lemma}[theorem]{Lemma} 
\newtheorem{corollary}[theorem]{Corollary}

\theoremstyle{definition}
\newtheorem{example}{Example}%
\newtheorem{remark}{Remark}%
\newtheorem{problem}{Problem}
\newtheorem{assumption}{Assumption}
\newtheorem{fact}{Fact}

\newtheorem{definition}{Definition}%

\usepackage[dvipsnames,svgnames]{xcolor}
\colorlet{MyBlue}{DodgerBlue!75!Black}
\newcommand{\revise}[1]{{\color{black}#1}}

\begin{document}

\title[Penalty dynamics]{Asymptotic behavior of penalty dynamics for constrained variational inequalities}

\author[1]{\fnm{Juan} \sur{Peypouquet}}\email{j.g.peypouquet@rug.nl}

\author*[1,2]{\fnm{Siqi} \sur{Qu}}\email{qu.siqi@uni-mannheim.de}

\author[2]{\fnm{Mathias} \sur{Staudigl}}\email{m.staudigl@uni-mannheim.de}

\affil*[2]{\orgdiv{Department of Mathematics}, \orgname{University of Mannheim}, \orgaddress{\street{B6 26}, \city{Mannheim}, \postcode{68159}, \country{Germany}}}

\affil[1]{\orgdiv{Faculty of Science and Engineering, Systems, Control and Optimization — Bernoulli Institute}, \orgname{Rijksuniversiteit Groningen}, \orgaddress{\street{Nijenborgh 9}, \city{Groningen}, \postcode{9747 AG}, \country{The Netherlands}}}


\abstract{We propose a comprehensive framework for solving constrained variational inequalities via various classes of evolution equations displaying multi-scale aspects. In an infinite-dimensional Hilbertian framework, the class of dynamical systems we propose combine Tikhonov regularization and exterior penalization terms in order to induce strong convergence of trajectories to least norm solutions in the constrained domain. Our construction thus unifies the literature on regularization methods and penalty-based dynamical systems. An extension to a full splitting formulation of the constrained domain is also provided, with associated weak convergence results involving the Attouch-Czarnecki condition.
}
\keywords{Constrained Variational Inequalities, Dynamical Systems, Penalty Dynamics, Tikhonov regularization}

\pacs[MSC Classification]{37C60, 37L05, 49J40, 90C25}

\maketitle

\section{Introduction}
\label{sec:intro}
Let $\scrH$ be a real Hilbert space with inner product $\inner{\cdot,\cdot}$ and corresponding norm $\norm{\cdot}$. 
\begin{problem} 
We are given 
\begin{itemize}
\item $\opA:\scrH\to2^{\scrH}$ a maximally monotone operator; 
\item $\opD:\scrH\to\scrH$ is monotone and $\frac{1}{\eta}$-Lipschitz; 
\item $\scrC\eqdef \zer(\opB)\neq\emptyset$ a nonempty closed convex set which can be represented as the zeros of a cocoercive operator $\opB:\scrH\to\scrH$. 
\end{itemize}
Our aim is to find $x\in\scrH$ such that 
\begin{equation}\label{eq:MI}\tag{P}
0\in \Phi(x)\eqdef \opA(x)+\opD(x)+\NC_{\scrC}(x).
\end{equation}
\end{problem}

This is a three-operator formulation of a general class of variational problems, in which a constrained equilibrium of the sum of two maximally monotone operators $\opA+\opD$ is requested over a closed convex domain $\scrC\subset\scrH$. In this paper we approach this problem via a combination of an exterior penalization method and Tikhonov regularization, essentially leading to a hierarchical equilibrium formulation. A key assumption in our approach is thus that the feasible set $\scrC$ admits a simple representation in terms of the set of zeros of another single-valued monotone operator $\opB$. This general formulation admits many applications in optimal control and optimization, in particular those problems featuring a hierarchical structure. We briefly describe some examples below, and provide a more detailed exposition on applications in Section \ref{sec:Applications}.

\begin{example}[Simple Bilevel Optimization]
A simple bilevel optimization problem is formulated as  
\begin{align*}
 & \min f(x) \\ 
\text{s.t.: } & x\in\scrC=\argmin\{\Psi(y):y\in\scrH\}
\end{align*}
where $\Psi$ is a convex and Fr\'{e}chet differentiable function, with $L$-Lipschitz continuous gradient, i.e. $\Psi\in\bC^{1,1}_{L}(\scrH)$. By Fermat's optimality condition 
$$
x\in\scrC\iff 0=\nabla \Psi(x).
$$
Hence, $\scrC=\zer(\nabla\Psi)$ defines the set of solutions to the lower-level optimization problem over which a minimizer of the function $f$ is searched for. 
These are simplest hierarchical optimization problems, in which the decision variables  in the lower and an upper level variational problems are completely decoupled. This class of problems has been studied extensively. See, for instance \cite{XuKim03,ogura2002non,Bot:2019aa} in an infinite-dimensional setting, or \cite{sabach2017first} in finite dimension. For a recent dynamical systems perspective with acceleration, we refer to \cite{boct2025accelerating} 
\end{example}

\begin{example}[Constrained Variational Inequalities]
Let $\scrC\subset\scrH$ be a set for which we know a convex function $\Psi\in\bC^{1,1}_{L_{\Psi}}(\scrH)$ having the properties that $\Psi\geq 0$ and $\scrC=\Psi^{-1}(0)$. Then, $\opB=\nabla\Psi$ is a maximally monotone operator whose zero set is $\scrC$. Given a mapping $\opD:\scrH\to\scrH$ and a proper convex and lower semi-continuous function $h:\scrH\to(-\infty,\infty]$, we search for a solution of the variational inequality of the second kind  
\begin{equation}
\text{Find $\bar{x}\in\scrC$ such that } \inner{\opD(\bar{x}),x-\bar{x}}+h(x)-h(\bar{x})\geq 0 \qquad\forall x\in\scrC.
\end{equation}
This is equivalent to the inclusion problem
$$
\text{Find $\bar{x}\in\scrC$ such that } 0\in \opD(\bar{x})+\partial h(\bar{x})+\NC_{\scrC}(x), \qquad\forall x\in\scrC.
$$
and arises frequently in optimal control problems \cite{BarbuControlVI,de2011optimal,de2016strong}. 
\end{example}

\subsection*{Our approach}
To tackle problem class \eqref{eq:MI}, we follow a recent trend in variational analysis by approaching the general resolution of problem \eqref{eq:MI} via dynamical systems derived from operator splitting methods featuring multiscale aspects. Our schemes can be considered as hybrid versions of penalty-based methods, inspired by \cite{Attouch:2010aa,AttCzarPey11,banert2015backward,noun2013forward,Bot:2016aa,Bot:2014aa}, and Tikhonov regularization of dynamical systems \cite{attouch1996dynamical,cominetti2008strong,Bot:2020aa}.
Specifically, if $\opD$ is cocoercive, we propose a forward-backward dynamical system of the form
\begin{equation}\tag{FB}\label{eq:FB}
\dot{x}(t)+x(t)=\res_{\lambda(t)\opA}\left(x(t)-\lambda(t)V_{t}(x)\right), 
\end{equation}
where $V_{t}:\scrH\to\scrH$ is a time-dependent vector field, defined as 
$$
V_{t}(x)\eqdef \opD(x)+\eps(t)x+\beta(t)\opB(x).
$$
The evolution equation \eqref{eq:FB} generalizes most of the first-order dynamical systems that have been studied so far in the literature. Indeed, if $\res_{\lambda(t)\opA}=\Id_{\scrH}$ (arising if $\opA=\{0\})$, then \eqref{eq:FB} simplifies to 
$$
\dot{x}(t)+\lambda(t)(\opD(x(t))+\eps(t)x(t)+\beta(t)\opB(x(t)))=0.
$$
In the full potential case, i.e. when $\opD=\nabla f$ and $\opB=\nabla\Psi$, for convex potentials $f,\Psi:\scrH\to\R$ having a Lipschitz continuous gradient, with $\scrC=\Psi^{-1}(0)=\argmin(\Psi)=\zer(\nabla\Psi)$, we obtain an exterior penalization of the gradient method, as studied in \cite{Peypouquet:2012aa}, augmented with Tikhonov regularization, for solving set-constrained convex minimization problems of the form
$$
\min f(x) \quad \text{s.t. }x\in\scrC.
$$
Setting $\opB=0$ (no penalty), the dynamics reduce to a gradient system with Tikhonov regularization, a class of dynamical systems which has received enormous attention in the literature (references \cite{Att96,AlvCab06,cominetti2008strong,Attouch:2010aa,BaiCom01,Cabot2005} give a partial overview on this literature). In the context of full splitting and maximal monotone operators instead of potentials, particular instances of \eqref{eq:FB} have been studied in \cite{BotCse17} without explicit constraints and Tikhonov regularization, in \cite{Bot:2016aa} for the case with a penalty term but no Tikhonov regularization, and in \cite{Bot:2020aa} in the context of Tikhonov regularization and without penalty terms. 

For large-scale equilibrium and minmax problems, the cocoercivity of the single-valued operator $\opD$ is typically an incompatible hypothesis. For such problems, we study an extragradient-based dynamical system in the spirit of Tseng \cite{TsengFBF}, featuring multi-scale aspects. This dynamical system is defined in terms of the projection-differential dynamical system 
\begin{equation}\tag{FBF}\label{eq:FBF}
\left\{\begin{array}{rcl}
p(t)&=&\res_{\lambda(t)\opA}\left(x(t)-\lambda(t)V_{t}(x(t))\right)\\
\dot{x}(t)&=&p(t)-x(t)+\lambda(t)[V_{t}(x(t))-V_{t}(p(t))].
\end{array}\right.
\end{equation}
The asymptotic properties of this evolution equation with Tikhonov regularization has been studied in \cite{Bot:2020aa}, while the penalty case has been investigated in \cite{Bot:2014aa}, although in discrete time. While we do not treat this problem directly here, it should be mentioned that our approach can be used to solve constrained min-max problems in infinite dimensions \cite{Dai24,FastSaddle25}. A thorough analysis of this important problem class in our multi-scale dynamical system framework is left for future research.

\subsection*{Contributions and related literature}
Our dynamical system combines exterior-penalty methods with Tikhonov regularization. With this hybrid construction, we generalize the pure penalization-based dynamical systems studied in \cite{AttCzarPey11,Peypouquet:2012aa}, who concentrate on the case where $\opB$ is the gradient of a convex function $\Psi:\scrH\to [0,\infty]$, satisfying $\scrC=\Psi^{-1}(0)=\argmin_{\scrH}(\Psi)$, and \cite{Bot:2014aa} who extended this to the monotone inclusion setting. Adding the Tikhonov term to the dynamical system allows us to enforce strong convergence to the least-norm solution. We borrow ideas from \cite{Bot:2020aa} to analyze the effects of the Tikhonov term, and extend this technique to deal with constrained variational inequalities. To the best of our knowledge this is a new result in the literature since the analysis of penalty dynamics, \emph{in tandem} with Tikhonov regularization, has not been done so far. In particular, this combined approach gains relevance in inverse problems and PDE constrained optimization; See \cite{doi:10.1137/23M1553170} for a recent approach in this direction. 

Our work is closely related to the hierarchical minimization formulation of Cabot \cite{Cabot2005}. He proposed an inexact proximal point method for the hierarchical minimization of a finitely many convex potentials. Specifically, in the case of three convex functions $f_{0},f_{1},f_{2}$, the method of \cite{Cabot2005} is designed for solving the nested optimization problems $S_{0}=\argmin_{\scrH}f_{0},S_{1}=\argmin_{S_{0}}f_{1}$ and $S_{2}=\argmin_{S_{1}}f_{2}$. Under time rescaling and constraint qualification, the exact version of the proximal point method of \cite{Cabot2005} thus becomes 
$$
-\frac{x_{n}-x_{n+1}}{\lambda_{n}}\in \beta_{n}\partial f_{0}+\partial f_{1}+\eps_{n} \partial f_{2}. 
$$
In our work, we are interested in continuous-time systems and general monotone operators. Both directions lie out of the scope of \cite{Cabot2005}.

In parallel to this work, a series of papers studied penalty methods and Tikhonov regularization from the lens of dynamical systems (separately). In particular, \cite{bot2024tikhonov} establishes an interesting connection between the Tikhonov regularization and the Halpern schemes and thereby shows the acceleration potential of the Tikhonov regularization. Inertial dynamics with Tikhonov regularization have been studied in \cite{attouch2018combining,BOT2024127689}, among many others. 
Penalty dynamics have been extended to second order in time systems in \cite{boct2017penalty,boct2018second}.\footnote{ 
Parts of the results reported in this paper are published in the conference proceeding \cite{10530399}. Besides giving detailed proofs of all results, which partly were missing from the proceedings, we extend the approach by a new splitting scheme with multiple penalty functions, and present applications and numerical examples.} 
 
\subsection*{Organization of the paper}
After briefly recalling some known facts from convex analysis and monotone operator theory, Section \ref{sec:prelims} presents a detailed analysis of the \emph{central paths}, which are curves parameterized by regularization variables, representing solutions to auxiliary monotone inclusion problems. In particular, we show that central paths are absolutely continuous, differentiable - an important ingredient in our proof building on Lyapunov analysis - and approximate the least-norm solution of \eqref{eq:MI}. Section \ref{sec:dynamics} is concerned with two dynamical systems intended to approximate solutions of \eqref{eq:MI}. The nature of these systems, of either forward-backward or forward-backward-forward type, depends on whether the operator $\opD$ is cocoercive. For each system, and under suitable assumptions on the regularization parameters, we show that every trajectory convergence strongly to the least norm solution of \eqref{eq:MI}. This is achieved by establishing a tracking property of the dynamics with respect to the central paths. Section \ref{sec:Multiscale} extends our analysis of the \eqref{eq:FB} dynamical systems to a domain decomposition problem, modeled via multiple penalty functions. Section \ref{sec:Applications} describes some scenarios in optimal control and non-linear analysis to which our method naturally applies. We also present implementations on a class of image deblurring problems, to illustrate the computational efficacy of the method.

\section{Preliminaries}
\label{sec:prelims}
We follow the standard notation from convex and non-smooth analysis (see e.g. \cite{BauCom16,PeyConvexBook}). Let $\scrH$ be a real Hilbert space with inner product $\inner{\cdot,\cdot}$ and associated norm $\norm{\cdot}$. The symbols $\wlim$ and $\to$ denote weak and strong convergence, respectively. For an extended-real valued function $f:\scrH\to\bar{\R}\eqdef\R\cup\{-\infty,\infty\}$, we denote by $\dom(f)=\{x\in\scrH\vert f(x)<\infty\}$ its effective domain and say that $f$ is proper if $\dom(f)\neq\emptyset$ and $f(x)\neq-\infty$ for all $x\in\scrH$. If $f$ is convex, we let $\partial f(x)\eqdef\{u\in\scrH\vert f(y)\geq f(x)+\inner{y-x,v}\; \forall y\in\scrH\}$ the subdifferential of $f$ at $x\in\dom(f)$. 

Let $\scrC\subseteq\scrH$ be a nonempty set. The indicator function $\delta_{\scrC}:\scrH\to\bar{\R}$, is the function satisfying $\delta_{\scrC}(x)=0$ if $x\in\scrC$ and $+\infty$ otherwise. The subdifferential of the indicator function is the normal cone 
\[
\NC_{\scrC}(x)\eqdef\left\{\begin{array}{ll} \{\xi\in\scrH\vert \inner{\xi,y-x}\leq 0\;\forall y\in\scrC \} & \text{if }x\in\scrC,\\ 
\emptyset & \text{else}.
\end{array}\right.
\]

For a set-valued operator $\opM:\scrH\to 2^{\scrH}$ we denote by $\gr(\opM)\eqdef \{(x,u)\in\scrH\times\scrH\vert u\in \opM(x)\}$ its graph, $\dom(\opM)\eqdef\{x\in\scrH\vert \opM(x)\neq\emptyset\}$ its domain, and by $\opM^{-1}:\scrH\to 2^{\scrH}$ its inverse, defined by 
$(u,x)\in \gr(\opM^{-1})\iff (x,u)\in\gr(\opM).$ We let $\zer(\opM)=\{x\in\scrH\vert 0\in \opM(x)\}$ denote the set of zeros of $\opM$. A operator $\opM$ is \emph{monotone} if $\inner{x-y,u-v}\geq 0$ for all $(x,u),(y,v)\in\gr(\opM)$. A monotone operator $\opM$ is maximally monotone if there exists no proper monotone extension of the graph of $\opM$ on $\scrH\times\scrH$. A single-valued operator $\opM:\scrH\to\scrH$ is \emph{cocoercive} (inverse strongly monotone) if there exists $\mu>0$ such that 
$$
\inner{\opM(x)-\opM(y),x-y}\geq\mu\norm{\opM(x)-\opM(y)}^{2}.
$$
A cocoercive operator is Lipschitz, but the converse is not true in general.
\begin{fact}\cite[Proposition 23.39]{BauCom16}
If $\opM$ is maximally monotone, then $\zer(\opM)$ is convex and closed.
\end{fact}

\begin{fact}\label{fact:angle}
If $\opM$ is maximally monotone, then 
\[p\in\zer(\opM)\iff \inner{u-p,w}\geq 0\quad\forall (u,w)\in\gr(\opM).\]
\end{fact}

The resolvent of $\opM$, $\res_{\opM}:\scrH\to 2^{\scrH}$ is defined by $\res_{\opM}=(\Id+\opM)^{-1}$. If $\opM$ is maximally monotone, then $\res_{\opM}$ is single-valued and maximally monotone. In particular, the resolvent of the normal cone mapping $\opM=\NC_{C}$ of a closed convex set $C\subset\scrH$ is the orthogonal projection $\Pi_{C}$. The support function of a closed set $C\subset\scrH$ is $\sigma_{C}(u)=\sup_{y\in C}\inner{y,u}$. Clearly, $p\in\NC_{C}(u)$ if, and only, if $\sigma_{C}(p)=\inner{p,u}$. According to \cite[Proposition 23.31]{BauCom16}, we have the relation 
\begin{equation}\label{eq:resolventContinuous}
\norm{\res_{\lambda M}(x)-\res_{\alpha M}(x)}\leq\abs{\lambda-\alpha}\cdot \norm{\opM_{\lambda}(x)},
\end{equation}
where $\opM_{\lambda}\eqdef \frac{1}{\lambda}(\Id_{\scrH}-\res_{\lambda\opM})$ denotes the \emph{Yosida approximation}, with parameter $\lambda$, of the maximally monotone operator $\opM$. 
The Fitzpatrick function associated to a monotone operator $\opM$ is the convex and lower semi-continuous function defined as 
\[
\varphi_{\opM}(x,u)=\sup_{(y,v)\in\gr(\opM)}\{\inner{x,v}+\inner{y,u}-\inner{y,v}\}.
\]

\begin{lemma}
Let $x,y,z\in\scrH$ and $\alpha\in\R$. Then  
\begin{align}\label{eq:3point}
&2\inner{x-y,z-y}=\norm{x-y}^{2}-\norm{x-z}^{2}+\norm{z-y}^{2}\\
&\norm{\alpha x+(1-\alpha)y}^{2}+\alpha(1-\alpha)\norm{x-y}^{2}=\alpha\norm{x}^{2}+(1-\alpha)\norm{y}^{2}
\label{eq:convexHilbert}
\end{align}
\end{lemma}

To assess the global asymptotic stability of the dynamical systems we consider, we recall the following central result (see e.g. \cite{abbas2014newton}, Lemma 5.1): 

\begin{lemma}\label{lem:Abbas}
    Suppose that $F:\R_{\geq0}\to\R$ is locally absolutely continuous and bounded below and that there exists $G\in L^{1}(\R_{\geq 0})$ such that
    $$
    \frac{\dif}{\dif t} F(t)\leq G(t)\qquad\text{a.e. }t\in\R_{\geq 0}. 
    $$
    Then $\lim_{t\to\infty}F(t)$ exists in $\R$. 
\end{lemma}
\subsection{Perturbed solutions and the central funnel}

We follow a {\it double penalization} approach to solve the constrained variational inequality \eqref{eq:MI}. We work under following standing assumption:

\begin{assumption}\label{ass:standing}
The main set of standing hypothesis in this paper is the following:
\begin{itemize}
\item[(i)]  $\opA+\NC_{\scrC}$ is maximally monotone and $\zer(\opA+\opD+\NC_{\scrC})\neq\emptyset$. 
\item[(ii)] $\opB:\scrH\to\scrH$ is $\mu$-cocoercive (for some $\mu>0)$, and coercive: 
$$
\lim_{\norm{x}\to\infty}\frac{\inner{x,\opB(x)}}{\norm{x}}=\infty
$$
\item[(iii)] $\opD:\scrH\to\scrH$ is maximally monotone and $\frac{1}{\eta}$-Lipschitz.
\end{itemize}
\end{assumption}

Since $\opA$ is maximally monotone and $\scrC$ is a nonempty convex and closed set, $\opA+\NC_{\scrC}$ is maximally monotone if additional regularity conditions are satisfied. A comprehensive list of conditions guaranteeing the maximal monotonicity of the sum of two maximally monotone operators can be found in \cite{borwein2010convex,BauCom16}. In particular, if $\dom(\opD)=\scrH$ and $\dom(\opB)=\scrH$ (which we assume throughout this manuscript), Corollary 25.5 of \cite{BauCom16} implies that the operators $\Phi=\opA+\opD+\NC_{\scrC}$ and $\opF_{n}=\opA+\opD+\beta_{n}\opB$ are maximally monotone. In particular, $\zer(\opA+\opD+\NC_{\scrC})$ is a nonempty closed convex set.

To ensure strong convergence, we follow the classical iterative Tikhonov regularization approach. However, besides strong convergence, we have to respect the constraints imposed on the variational inequality. To enforce these constraints, we augment our operator by a time-varying penalty term which regulates over time the importance we attach to constraint violation of the generated trajectory. The combination of these two dynamic effects leads us to study a family of \emph{auxiliary problems}, formulated as:
\begin{problem}
Given $(\eps,\beta)\in\R^{2}_{\geq0}$ find $x\in\scrH$ such that 
\begin{equation}\label{eq:MIauxiliary}
0\in \Phi_{\eps,\beta}(x)\eqdef (\opA+\opD+\eps\Id_{\scrH}+\beta \opB)(x).
\end{equation}
\end{problem}
The logic behind this model is as follows: Suppose $0\in\Phi_{\eps,\beta}(\hat x_{\eps,\beta})$. On the one hand, the presence of the Tikhonov term $\eps\Id_{\scrH}$ will attract the vector $\hat x_{\eps,\beta}$ towards the origin and, as $\eps\to 0$, $\hat x_{\eps,\beta}$ will approach the set of solutions of \eqref{eq:MI}, selecting its least-norm element. On the other hand, the penalization term $\beta \opB$ will force $\opB(\hat x_{\eps,\beta})\to 0$, as $\beta\to\infty$. As a result, the limit of $\hat x_{\eps,\beta}$ will be a zero of $B$, thus an element of $\scrC$.

In the remainder of this section, we show that solutions of \eqref{eq:MIauxiliary} do approximate the least-norm solution of \eqref{eq:MI}, as $(\varepsilon,\beta)\in\R^2_{ \geq 0}$, and establish some important structural properties of this approximation. For $(\eps,\beta)\in\R^{2}_{\geq 0}$, we define the mapping $V_{\eps,\beta}:\scrH\to\scrH$ by 
\begin{equation} \label{E:V_e_b}
    V_{\eps,\beta}(x)\eqdef \opD(x)+\eps x+\beta \opB(x)
\end{equation}
for all $x\in\scrH$. Assumption \ref{ass:standing} immediately implies the following regularity properties of the parameterized family $(V_{\eps,\beta})_{(\eps,\beta)\in\R_{\geq 0}^{2}}.$
\begin{lemma}\label{lem:LipschitzV}
For each $(\eps,\beta)\in\R^{2}_{\geq 0}$, the mapping $x\mapsto V_{\eps,\beta}(x)$ is Lipschitz continuous with modulus $L_{\varepsilon,\beta}\eqdef \frac{1}{\eta}+\eps+\frac{\beta}{\mu}$. Additionally, if $\eps>0$,  $V_{\eps,\beta}(x)$ is $\eps$-strongly monotone. Furthermore, $\dom(V_{\eps,\beta})=\scrH$.
\end{lemma}
Hence, the auxiliary problems \eqref{eq:MIauxiliary} constitute a family of strongly monotone inclusions. Since $\Phi_{\eps,\beta}$ is strongly monotone, the set $\zer(\Phi_{\eps,\beta})$ is a singleton whose unique element we denote by $\bar{x}(\eps,\beta)$. The function $(\eps,\beta)\mapsto\bar{x}(\eps,\beta)$ maps the positive quadrant of $\R^2$ to a region in $\scrH$, which we call the {\it central funnel} for the reasons we will explain in the following.

\subsection{Central paths approximate the least-norm zero of $\Phi$}
\label{sec:Centralpath}
Given absolutely continuous functions $\eps,\beta:\R_{\geq 0}\to\R_{>0}$, such that $\lim_{t\to\infty}\eps(t)=0$ and $\lim_{t\to\infty}\beta(t)=\infty$, the curve $t\mapsto \bar{x}(\eps(t),\beta(t))$ is called the {\it central path} for the approximation of Problem \eqref{eq:MI} given by \eqref{eq:MIauxiliary}. Similarly, given positive sequences $(\eps_{n})_{n\in\N}$ and $(\beta_{n})_{n\in\N}$, such that $\eps_{n}\to 0$ and $\beta_{n}\to\infty$, the sequence $\big(\bar{x}(\eps_{n},\beta_{n})\big)_{n\in\N}$ is a {\it (discrete) central path}. As shown below, every central path converges strongly to the least-norm zero of $\Phi$, which motivates our choice of the word funnel.

\revise{
\begin{proposition}\label{prop:asymptotics}
Let $(\eps_{n})_{n\in\N},(\beta_{n})_{n\in\N}$ be sequences in $\R_{>0}$ such that $\eps_{n}\to 0,\beta_{n}\to\infty$, and $\lim_{n\to\infty}\eps_{n}\beta_{n}=\infty$. Suppose in addition that for $z^{*}=\Pi_{\zer(\Phi)}(0)$ and some $p^{*}\in\NC_{\scrC}(z^{*})$ satisfying $-p^{*}\in\opA(z^{*})+\opD(z^{*})$ 
\begin{equation}\label{eq:AC-Strong}
\lim_{n\to\infty}\frac{\beta_{n}}{\eps_{n}}\left[\varphi_{\opB}(z^{*},\frac{p^{*}}{\beta_{n}})-\sigma_{\scrC}(p^{\ast}/\beta_{n})\right]=0. 
\end{equation}
Then $\bar{x}(\eps_{n},\beta_{n})\to\Pi_{\zer(\Phi)}(0)$ as $n\to\infty$. 
\end{proposition}
}

We prove this proposition via a sequence of Lemmas. 
\begin{lemma}\label{lem:limitfeasible}
Let $(\eps_{n})_{n\in\N},(\beta_{n})_{n\in\N}$ be sequences in $\R_{>0}$ such that $\eps_{n}\to 0,\beta_{n}\to\infty$, and $\lim_{n\to\infty}\eps_{n}\beta_{n}=\infty$. Denote by $(\bar{x}_{n})_{n\in\N}$ the discrete central path obtained by $\bar{x}_{n}=\bar{x}(\eps_{n},\beta_{n})$. Any weak limit point of $(\bar{x}_{n})_{n\in\N}$ is contained in $\scrC$. 
\end{lemma}
\begin{proof}
 For every $n\in\N$, define $\bar{x}_{n}\equiv \bar{x}(\eps_{n},\beta_{n})$, the unique root of the set-valued mapping $\Phi_{n}\equiv\Phi_{\eps_{n},\beta_{n}}:\scrH\to 2^{\scrH}$. To simplify the notation, we define the set-valued operator $\opF_{n}:\scrH\to 2^{\scrH}$ by
$$   
\opF_{n}\eqdef \opA+\opD+\beta_{n}\opB. 
$$
Hence, for each $n\geq 1$, we have the inclusion 
\begin{align*}
0\in &\opF_{n}(\bar{x}_{n})+\eps_{n}\bar{x}_{n}\iff (1-\eps_{n})\bar{x}_{n}\in (\Id_{\scrH}+\opF_{n})(\bar{x}_{n})\\
& \iff \bar{x}_{n}=\res_{\opF_{n}}((1-\eps_{n})\bar{x}_{n})
\end{align*}
At the same time, for every $z\in\zer(\opA+\opD+\NC_{\scrC})$, it holds that $z\in\scrC$, and therefore $\beta_{n}\opB(z)=0$ for all $n\geq 1$. This implies that there exists $p\in\NC_{\scrC}(z)$ such that for all $n\in\N$ 
$$
z-p\in(\Id_{\scrH}+\opF_{n})(z)\iff z=\res_{\opF_{n}}(z-p).
$$
Since $\opF_{n}$ is maximally monotone, $\res_{\opF_{n}}$ is firmly non-expansive. Hence, 
\begin{align*}
\norm{\bar{x}_{n}-z}^{2}&\leq \inner{\res_{\opF_{n}}((1-\eps_{n}\bar{x}_{n}))- \res_{\opF_{n}}(z-p), (1-\eps_{n})\bar{x}_{n}-(z-p)}\\
&=\inner{\bar{x}_{n}-z,(1-\eps_{n})\bar{x}_{n}-(z-p)}
\end{align*}
Rearranging yields the inequality 
\begin{equation}\label{eq:finalstep}
\inner{\bar{x}_{n}-z,p}\geq\eps_{n}\inner{\bar{x}_{n},\bar{x}_{n}-z}\geq \eps_{n}\norm{\bar{x}_{n}}(\norm{\bar{x}_{n}}-\norm{z}).
\end{equation}
$\res_{\opF_{n}}$ is also non-expansive, so that 
\begin{align*}
\norm{\bar{x}_{n}-z} & \leq \norm{(1-\eps_{n})\bar{x}_{n}-z+p}\\
&=\norm{(1-\eps_{n})(\bar{x}_{n}-z)+p-\eps_{n}z}\\ 
&\leq (1-\eps_{n})\norm{\bar{x}_{n}-z}+\norm{p-\eps_{n}z}. 
\end{align*}
Hence, 
$$
\eps_{n}\norm{\bar{x}_{n}-z}=\norm{\eps_{n}\bar{x}_{n}-\eps_{n}z}\leq \norm{p-\eps_{n}z} \qquad \forall n\geq 1.
$$
It follows that $(\eps_{n}\bar{x}_{n})_{n\in\N}$ is bounded.\\ 

We can also reorganize the equilibrium condition defining the sequence $\bar{x}_{n}$ and the point $z$ as 
\begin{align}
-\eps_{n}\bar{x}_{n}-\beta_{n}\opB(\bar{x}_{n})&\in \opA(\bar{x}_{n})+\opD(\bar{x}_{n}), \label{eq:xnsolve} \text{ and } \\ 
-p&\in \opA(z)+\opD(z) \label{eq:zsolve}
\end{align}
Monotonicity of $\opA+\opD$ implies 
\begin{equation}\label{LimitB}
\beta_{n}\inner{\opB(\bar{x}_{n}),z-\bar{x}_{n}}\geq \inner{\eps_{n}\bar{x}_{n},\bar{x}_{n}-z}+\inner{p,z-\bar{x}_{n}}.
\end{equation}
The $\mu$-co-coercivity of the penalty operator $\opB$ yields 
$$
\mu\norm{\opB(\bar{x}_{n})}^{2}\leq \inner{p-\eps_{n}\bar{x}_{n},\frac{\eps_{n}(\bar{x}_{n}-z)}{\beta_{n}\eps_{n}}}.
$$
If $\eps_{n}\to 0$ and $\eps_{n}\beta_{n}\to\infty$ as $n\to\infty$, we conclude from the above relation that
$$
\limsup_{n\to\infty}\norm{\opB(\bar{x}_{n})}^{2}\leq  0.
$$
Since the operator $\opB$ is assumed to be coercive, the sequence $(\bar{x}_{n})_{n}$ is bounded.  Let $\bar{x}_{\infty}$ be a weak accumulation point of $(\bar{x}_{n})_{n\in\N}$. Pick a subsequence $(x_{n_{j}})$ with $x_{n_{j}}\wlim \bar{x}_{\infty}$ as $j\rightarrow \infty$. \eqref{LimitB} implies 
$$
0\leq \inner{\opB(\bar{x}_{n_{j}}),\bar{x}_{n_{j}}-z}\leq \frac{\eps_{n_{j}}}{\beta_{n_{j}}}\inner{\bar{x}_{n_{j}},z-\bar{x}_{n_{j}}}+\frac{1}{\beta_{n_{j}}}\inner{p,\bar{x}_{n_{j}}-z}. 
$$
Hence, 
$$
\limsup_{n\to \infty}\inner{\opB(\bar{x}_{n}),\bar{x}_{n}-z}=0, 
$$
and for all $w\in\scrH$, we deduce 
$$
\inner{\opB(w),w-\bar{x}_{\infty}}=\lim_{j\to\infty}\inner{\opB(w),w-\bar{x}_{n_{j}}}\geq \lim_{j\to\infty}\inner{\opB(\bar{x}_{n_{j}}),w-\bar{x}_{n_{j}}}=0.
$$
Fact \ref{fact:angle}, implies $\bar{x}_\infty\in\scrC$. 
\end{proof}
Next to the feasibility of weak accumulation points of $(\bar{x}_{n})_{n\in\N}$, we want these limit points to belong to the set of solutions to \eqref{eq:MI}. 
\begin{lemma}\label{lem:limitSolution}
Let $(\eps_{n})_{n\in\N},(\beta_{n})_{n\in\N}$ be sequences in $\R_{>0}$ such that $\eps_{n}\to 0,\beta_{n}\to \infty$, and $\lim_{n\to\infty}\eps_{n}\beta_{n}=\infty$. Any weak limit point of the central path $(\bar{x}_{n})_{n\in\N}$ is contained in $\zer(\Phi)$. In particular, $\liminf_{n\to\infty}\norm{\bar{x}_{n}}\geq \inf\{\norm{z}:\;z\in\zer(\opA+\opD+\NC_{\scrC})\}$.
\end{lemma}
\begin{proof}
Let $(u,w)\in\gr(\Phi)$ arbitrary. Then, there exists $\xi\in\NC_{\scrC}(u)$ such that 
$$
w-\xi-\opD(u)\in \opA(u).
$$
Together with \eqref{eq:xnsolve}, the monotonicity of $\opA$ gives 
\[
\inner{-\eps_{n}\bar{x}_{n}-\beta_{n}\opB(\bar{x}_{n})-\opD(\bar{x}_{n})-w+\xi+\opD(u),\bar{x}_{n}-u}\geq 0.
\]
Whence, 
\begin{align*}
\inner{w,u-\bar{x}_{n}}&\geq \eps_{n}\inner{\bar{x}_{n},\bar{x}_{n}-u}+\beta_{n}\inner{\opB(\bar{x}_{n}),\bar{x}_{n}-u}\\
&+\inner{\opD(\bar{x}_{n})-\opD(u),\bar{x}_{n}-u}+\inner{\xi,u-\bar{x}_{n}}\\
&\geq \eps_{n}\inner{\bar{x}_{n},\bar{x}_{n}-u}+\inner{\xi,u-\bar{x}_{n}},
\end{align*}
where we have used the monotonicity of $\opB$ and $\opD$, as well as $\opB(u)=0$. If $\bar x_\infty$ is a weak accumulation point of $(\bar x_n)$, then 
\[
\inner{w,u-\bar x_\infty}\geq \inner{\xi,u-\bar{x}_\infty}. 
\]
The right-hand side is nonnegative since $\xi\in\NC_{\scrC}(u)$ and $\bar x_\infty\in \scrC$, thanks to Lemma \ref{lem:limitfeasible}. Using Fact \ref{fact:angle}, this means that $\bar x_\infty \in \zer(\Phi)$. In particular, $\norm{\bar{x}_{\infty}}\geq \inf\{\norm{z}:z\in\zer(\opA+\opD+\NC_{\scrC})\}$ for every weak limit point $\bar{x}_{\infty}$ of $(\bar{x}_{n})_{n\in\N}$.
\end{proof}
With the help of Lemmas \ref{lem:limitfeasible} and \ref{lem:limitSolution}, we can proof Proposition \ref{prop:asymptotics}.

\revise{
\begin{proof}[Proof of Proposition \ref{prop:asymptotics}]
By definition of the sequence $(\bar{x}_{n})$, we have 
$$
-p^{*}\in\opA(z^{*})+\opD(z^{*})\text{ and }-\eps_{n}\bar{x}_{n}-\beta_{n}\opB(\bar{x}_{n})\in\opA(\bar{x}_{n})+\opD(\bar{x}_{n}). 
$$
Monotonicity gives 
\begin{align*}
\eps_{n}\inner{\bar{x}_{n},\bar{x}_{n}-z^{*}}&\leq \inner{p^{*},\bar{x}_{n}-z^{*}}-\beta_{n}\inner{\opB(\bar{x}_{n}),\bar{x}_{n}-z^{*}}\\
&=\beta_{n}\left[\inner{\opB(\bar{x}_{n}),z^{*}}+\inner{p^{*}/\beta_{n},\bar{x}_{n}}-\inner{\opB(\bar{x}_{n},\bar{x}_{n}}-\inner{p^{*}/\beta_{n},z^{*}}\right]\\
&\leq \beta_{n}\left[\varphi_{\opB}(z^{*},p^{*}/\beta_{n})-\sigma_{\scrC}(p^{*}/\beta_{n})\right]
\end{align*}
Hence, 
\begin{equation}\label{eq:mainrelation}
\norm{\bar{x}_{n}}^{2}-\inner{\bar{x}_{n},z^{*}}\leq \frac{\beta_{n}}{\eps_{n}}\left[\varphi_{\opB}(z^{*},p^{*}/\beta_{n})-\sigma_{\scrC}(p^{*}/\beta_{n})\right]\equiv r_{n}. 
\end{equation}
Using \eqref{eq:AC-Strong}, we obtain 
$$
\limsup_{n\to\infty}\norm{\bar{x}_{n}}^{2}\leq \limsup_{n\to\infty}\inner{\bar{x}_{n},z^{*}}. 
$$
Consider a weakly converging subsequence $(\bar{x}_{n_{j}})$ with $ \bar{x}_{n_{j}}\wlim x_{\infty}$. By Lemma \ref{lem:limitSolution}, we know that $x_{\infty}\in\zer(\Phi)$. Evaluating relation \eqref{eq:mainrelation} along this weakly converging subsequence, we obtain 
$$
\norm{\bar{x}_{n_{j}}}^{2}\leq \inner{\bar{x}_{n_{j}},z^{*}}+r_{n_{j}}
$$
Let $n_{j}\to\infty$, we obtain 
$$
\norm{\bar{x}_{\infty}}^{2}\leq \inner{\bar{x}_{\infty},z^{*}}\leq\norm{\bar{x}_{\infty}}\cdot\norm{z^{*}}.
$$
Hence, $\norm{\bar{x}_{\infty}}\leq\norm{z^{*}}$. But $\bar{x}_{\infty}\in\zer(\Phi)$ and $z^{*}=\Pi_{\zer(\Phi)}(0)$, so that $\norm{z^{*}}\leq\norm{\bar{x}_{\infty}}$. Therefore, $\norm{\bar{x}_{\infty}}=\norm{z^{*}}$. By uniqueness of the least-norm element over the closed convex set $\zer(\Phi)$, we conclude $\bar{x}_{\infty}=z^{*}$. Finally, we obtain from \eqref{eq:mainrelation} 
$$
\limsup_{n\to\infty}\norm{\bar{x}_{n}}^{2}\leq \norm{z^{*}}^2. 
$$
The weak lower-semi continuity of norms yields the reverse inequality. Hence, $\norm{\bar{x}_{n}}\to\norm{z^{*}}$.
\end{proof}

\begin{remark}\label{rem:errorbound}
Condition \eqref{eq:AC-Strong} holds under several standard assumptions.
In particular, although it does not require an error-bound assumption,
such an assumption provides a convenient criterion for its verification.
A sufficient condition is the quadratic estimate
$$
\varphi_{\opB}(z,u)-\sigma_{\scrC}(u)\leq K\norm{u}^{2},\qquad z\in\scrC,
$$
for some $K>0$. Indeed, taking $z=z^*$ and $u=p^*/\beta_n$ yields
$$
\frac{\beta_n}{\eps_n}\left[\varphi_{\opB}\left(z^*,\frac{p^*}{\beta_n}\right)-\sigma_{\scrC}\left(\frac{p^*}{\beta_n}\right)\right]\leq\frac{K\norm{p^*}^{2}}{\eps_n\beta_n},
$$
which converges to zero since $\eps_n\beta_n\to\infty$.

More generally, estimates of this type can be obtained in structured settings from suitable error-bound conditions. For instance, if a proper, convex, lower semicontinuous function $f$ satisfies the H\"olderian growth condition
$$
\frac{\tau}{\rho}\dist(x,\argmin f)^\rho \leq f(x)-\min f, \qquad \rho\in(1,2],
$$
then, denoting by $\rho^*$ the conjugate exponent,
$$
0\leq f^*(u)-\sigma_{\argmin f}(u)\leq\frac{\tau^{1-\rho^*}}{\rho^*}\norm{u}^{\rho^*}.
$$

Thus, whenever the Fitzpatrick gap associated with $\opB$ can be controlled by such a Fenchel gap, the H\"olderian error bound provides a sufficient criterion for \eqref{eq:AC-Strong}.
\end{remark}
}

The consequence for the continuous central paths is the following:
\begin{corollary} \label{cor:asymptotics}
    Let $\eps,\beta:\R_{\geq 0}\to\R_{>0}$ be absolutely continuous functions such that $\lim_{t\to\infty}\eps(t)=0$ and $\lim_{t\to\infty}\beta(t)=\infty$. Then, $\bar{x}(\eps(t),\beta(t))\to\Pi_{\zer(\Phi)}(0)$ as $t\to\infty$.
\end{corollary}
%
%
We next prove that $(\eps,\beta)\mapsto \bar{x}(\eps,\beta)$ is a locally Lipschitz continuous function. This will be a fundamental result in the Lyapunov analysis of the dynamical systems.  
\begin{proposition}\label{prop:solutionmap}
The solution mapping $(\eps,\beta)\mapsto \bar{x}(\eps,\beta)$ is locally Lipschitz continuous. In particular, for all $\sigma_{1}=(\eps_{1},\beta_{1})$ and $\sigma_{2}=(\eps_{2},\beta_{2})$, we have 
\begin{equation}
\norm{\bar{x}(\sigma_{2})-\bar{x}(\sigma_{1})}\leq \frac{\ell}{\eps_{1}}\left(\abs{\beta_{2}-\beta_{1}}+\abs{\eps_{2}-\eps_{1}}\right),
\end{equation}
where $r\eqdef\inf\{\norm{x}:x\in\zer(\Phi)\}$, and $\ell\eqdef \max\{r,\sup_{x\in\B(0,r)}\norm{\opB(x)}\}$. 
\end{proposition}
\begin{proof}
Fix $\beta>0$ and pick $\eps_{1},\eps_{2}>0$, so that $z_{1}=\bar{x}(\eps_{1},\beta)$ and $z_{2}=\bar{x}(\eps_{2},\beta)$, denote the corresponding unique solutions of \eqref{eq:MIauxiliary}. Denoting by $\opD_{\eps}\eqdef \opD+\eps \Id_{\scrH}$, we obtain  
$$
-\opD_{\eps_{1}}(z_{1})-\beta\opB(z_{1})\in \opA(z_{1}),\text{ and }-\opD_{\eps_{2}}(z_{2})-\beta \opB(z_{2})\in \opA(z_{2}).
$$
Since $\opA$ is maximally monotone, we have 
$$
\inner{z_{1}-z_{2},-\opD_{\eps_{1}}(z_{1})-\beta\opB(z_{1})+\opD_{\eps_{2}}(z_{2})+\beta\opB(z_{2})}\geq 0.
$$
Since $\opD$ and $\opB$ are both maximally monotone, we conclude $\inner{\eps_{1}z_{1}-\eps_{2}z_{2},z_{1}-z_{2}}\leq 0$. Assume first that $\eps_{2}>\eps_{1}$. Then
\[
0\geq \inner{\eps_{1}z_{1}-\eps_{2}z_{2},z_{1}-z_{2}}=\eps_{1}\norm{z_{1}-z_{2}}^{2}+(\eps_{1}-\eps_{2})\inner{z_{2},z_{1}-z_{2}}, 
\]
which means $(\eps_{2}-\eps_{1})\inner{z_{2},z_{1}-z_{2}}\geq\eps_{1}\norm{z_{1}-z_{2}}^{2}$. By Cauchy-Schwarz, 
$$
(\eps_{2}-\eps_{1})\norm{z_{2}}\cdot\norm{z_{1}-z_{2}}\geq \eps_{1}\norm{z_{1}-z_{2}}^{2}, 
$$
so that 
$$
\norm{z_{2}-z_{1}}\leq\frac{\eps_{2}-\eps_{1}}{\eps_{1}}\norm{z_{2}}.
$$
Next, assuming $\eps_{1}>\eps_{2}$. Then, interchanging the labels in the above inequality, we get 
$$
\norm{z_{2}-z_{1}}\leq\frac{\eps_{1}-\eps_{2}}{\eps_{2}}\norm{z_{1}}.
$$
Hence, $\norm{\bar{x}(\eps_{1},\beta)-\bar{x}(\eps_{2},\beta)}\leq \frac{\abs{\eps_{2}-\eps_{1}}}{\min\{\eps_{1},\eps_{2}\}}\max\{\norm{z_{1}},\norm{z_{2}}\}$. This shows that $\eps\mapsto \bar{x}(\eps,\beta)$ is locally Lipschitz. 

Now, fix $\eps>0$ and let $\beta_{1},\beta_{2}>0$. Denote $z_{1}=\bar{x}(\eps,\beta_{1})$ and $z_{2}=\bar{x}(\eps,\beta_{2})$. By definition, we have 
$$
-\opD_{\eps}(z_{1})-\beta_{1} \opB(z_{1})\in \opA(z_{1}),\text{ and }-\opD_{\eps}(z_{2})-\beta_{2} \opB (z_{2})\in \opA(z_{2}).
$$
It follows that $\beta_{2}\inner{B(z_{2}),z_{1}-z_{2}}-\beta_{1}\inner{B(z_{1}),z_{1}-z_{2}}\geq\eps\norm{z_{1}-z_{2}}^{2}.$ Now, if $\beta_{2}>\beta_{1}$, then  
$(\beta_{2}-\beta_{1})\inner{B(z_{1}),z_{1}-z_{2}}+\beta_{2}\inner{B(z_{2})-B(z_{1}),z_{1}-z_{2}}\geq \eps\norm{z_{1}-z_{2}}^{2}$. Using the monotonicity of $\opB$, we conclude 
$$
\norm{z_{1}-z_{2}}\leq \frac{\beta_{2}-\beta_{1}}{\eps}\norm{\opB(z_{1})}.
$$
If $\beta_{1}>\beta_{2}$, we repeat the above computation, and obtain 
$$
\norm{z_{1}-z_{2}}\leq \frac{\beta_{1}-\beta_{2}}{\eps}\norm{\opB(z_{2})}.
$$
This yields $\norm{z_{1}-z_{2}}\leq \frac{\abs{\beta_{1}-\beta_{2}}}{\eps}\max\{\norm{\opB(z_{1})},\norm{\opB(z_{2})}\}$, which shows that $\beta\mapsto\bar{x}(\eps,\beta)$ is locally Lipschitz, for all $\eps>0$. 

Next, we show the Lipschitz continuity of the bivariate map $(\eps,\beta)\mapsto\bar{x}(\eps,\beta)$. Let $\sigma_{1}\eqdef (\eps_{1},\beta_{1})$ and $\sigma_{2}\eqdef(\eps_{2},\beta_{2})$ with corresponding solutions $\bar{x}(\sigma_{1})$ and $\bar{x}(\sigma_{2})$. By definition of these points, we have 
\[
-V_{\sigma_{1}}(\bar{x} (\sigma_{1}))\in \opA(\bar{x}(\sigma_{1})),\text{ and }-V_{\sigma_{2}}(\bar{x}(\sigma_{2}))\in \opA(\bar{x}(\sigma_{2})).
\]
Hence, 
\[
\inner{\opD_{\eps_{2}}(\bar{x}(\sigma_{2})+\beta_{2}\opB(\bar{x}(\sigma_{2}))-\opD_{\eps_{1}}(\bar{x}(\eps_{1},\beta_{1})-\beta_{1}\opB(\bar{x}(\eps_{1},\beta_{1})),\bar{x}(\eps_{1},\beta_{1})-\bar{x}(\sigma_{2})}\geq 0.
\]
Rearranging, we obtain 
\begin{align*}
\inner{\eps_{2}\bar{x}(\sigma_{2})-\eps_{1}\bar{x}(\sigma_{1}),\bar{x}(\sigma_{1})-\bar{x}(\sigma_{2})}&\geq \beta_{1}\inner{\opB(\bar{x}(\sigma_{1})),\bar{x}(\sigma_{1})-\bar{x}(\sigma_{2})}\\
&+\beta_{2}\inner{\opB(\bar{x}(\sigma_{2})),\bar{x}(\sigma_{2})-\bar{x}(\sigma_{1})}.
\end{align*}
This gives 
\begin{align*}
\eps_{1}\inner{\bar{x}(\sigma_{2})-\bar{x}(\sigma_{1}),\bar{x}(\sigma_{1})-\bar{x}(\sigma_{2})}&\geq\beta_{1}\inner{\opB(\bar{x}(\sigma_{1})),\bar{x}(\sigma_{1})-\bar{x}(\sigma_{2})}\\
&+\beta_{2}\inner{\opB(\bar{x}(\sigma_{2})),\bar{x}(\sigma_{2})-\bar{x}(\sigma_{1})}\\
&-(\eps_{2}-\eps_{1})\inner{\bar{x}(\sigma_{2}),\bar{x}(\sigma_{1})-\bar{x}(\sigma_{2})}.
\end{align*}
Hence, 
\begin{align*}
\eps_{1}\norm{\bar{x}(\sigma_{1})-\bar{x}(\sigma_{2})}^{2}&\leq\beta_{1}\inner{\opB(\bar{x}(\sigma_{1})),\bar{x}(\sigma_{2})-\bar{x}(\sigma_{1})}\\
&+\beta_{2}\inner{\opB(\bar{x}(\sigma_{2})),\bar{x}(\sigma_{1})-\bar{x}(\sigma_{2})}\\
&+(\eps_{2}-\eps_{1})\inner{\bar{x}(\sigma_{2}),\bar{x}(\sigma_{1})-\bar{x}(\sigma_{2})}\\
&=(\beta_{1}-\beta_{2})\inner{\opB(\bar{x}(\sigma_{1})),\bar{x}(\sigma_{2})-\bar{x}(\sigma_{1})}\\
&+\beta_{2}\inner{\opB(\bar{x}(\sigma_{2}))-\opB(\bar{x}(\sigma_{1})),\bar{x}(\sigma_{1})-\bar{x}(\sigma_{2})}\\
&+(\eps_{2}-\eps_{1})\inner{\bar{x}(\sigma_{2}),\bar{x}(\sigma_{1})-\bar{x}(\sigma_{2})}\\
&\leq \abs{\beta_{2}-\beta_{1}}\norm{\opB(\bar{x}(\sigma_{1}))}\cdot\norm{\bar{x}(\sigma_{2})-\bar{x}(\sigma_{1})}\\
&+\abs{\eps_{2}-\eps_{1}}\norm{\bar{x}(\sigma_{2})}\cdot\norm{\bar{x}(\sigma_{2})-\bar{x}(\sigma_{1})}.
\end{align*}
We thus finally arrive at the estimate
\begin{equation} \label{eq:solutionmap}
    \norm{\bar{x}(\sigma_{2})-\bar{x}(\sigma_{1})}\leq \frac{\abs{\beta_{2}-\beta_{1}}}{\eps_{1}}\norm{\opB(\bar{x}(\sigma_{1}))}+\frac{\abs{\eps_{2}-\eps_{1}}}{\eps_{1}}\norm{\bar{x}(\sigma_{2})}.
\end{equation}
From the proof of Step (i) of the proof of Proposition \ref{prop:asymptotics}, we deduce that $\norm{\bar{x}(\eps,\beta)}\leq\inf\{\norm{x}:x\in\zer(\Phi)\}\eqdef r$. Hence, defining 
$\ell\eqdef \max\{\sup_{x\in\B(0,r)}\norm{\opB(x)},r\}$, the claim follows.
\end{proof}

\subsubsection{Differentiability of central paths}
We now assume that the parameters $(\eps,\beta)$ are defined in terms of real-valued functions $\eps,\beta:\R_{\geq0}\to\R_{>0}$. In terms of these functions, we define the time-dependent vector field
$$
V_t:\R_{\geq 0}\times\scrH\to\scrH,\quad t\mapsto V_{t}(x)\equiv V_{\eps(t),\beta(t)}(x).
$$
For each $t$, we obtain the unique solution $\bar{x}(t)\equiv\bar{x}(\eps(t),\beta(t))\in \zer(\opA+V_{t})$.  From Lemma \ref{lem:LipschitzV}, we know that $V_{t}:\scrH\to\scrH$ is $\eps(t)$-strongly monotone and $L(t)$-Lipschitz continuous, where 
\begin{equation}\label{eq:Lt}
L(t)\eqdef L_{\eps(t),\beta(t)}= \frac{1}{\eta}+\eps(t)+\frac{\beta(t)}{\mu}.
\end{equation}

\begin{assumption}\label{ass:params}
The functions $t\mapsto \eps(t),t\mapsto\beta(t)$ are absolutely continuous and satisfy the following properties:
\begin{itemize}
\item[(i)] $t\mapsto \eps(t)$ non-increasing and $\lim_{t\to\infty}\eps(t)=0$;
\item[(ii)] $t\mapsto\beta(t)$ non-decreasing and $\lim_{t\to\infty}\beta(t)=\infty$;
\item[(iii)] $\lim_{t\to\infty}\beta(t)\eps(t)=\infty.$
\end{itemize}
\end{assumption}

\begin{lemma}
\label{lem:CP}
Consider the central path $(\bar{x}(t))_{t\geq {0}}$. Under Assumption \ref{ass:params}, the curve $t\mapsto \bar{x}(t)$ is almost everywhere differentiable, with 
\begin{equation}\label{eq:boundCP}
\norm{\frac{\dif}{\dif t}{\bar{x}(t)}}\leq \frac{\dot{\beta}(t)}{\eps(t)}\norm{\opB(\bar{x}(t))}+\frac{\abs{\dot{\eps}(t)}}{\eps(t)}\norm{\bar{x}(t)}\qquad\text{a.e. }t\geq 0.
\end{equation}
\end{lemma}
\revise{
\begin{proof}
Let $T>0$. By Proposition \ref{prop:solutionmap}, for every
$0\leq t_{1}<t_{2}\leq T$, we have
\begin{align*}
\norm{\bar{x}(t_{2})-\bar{x}(t_{1})}&\leq\frac{\ell}{\eps(t_{1})}\left(\abs{\beta(t_{2})-\beta(t_{1})}+\abs{\eps(t_{2})-\eps(t_{1})}\right)\\
&\leq\frac{\ell}{\eps(T)}\int_{t_{1}}^{t_{2}}\left(\abs{\dot{\beta}(s)}+\abs{\dot{\eps}(s)}
\right)\dif s,
\end{align*}
where we used the fact that $\eps$ is non-increasing. Since $\eps$ and
$\beta$ are absolutely continuous, it follows that the mapping $t\mapsto\bar{x}(t)$ is locally absolutely continuous, and hence strongly differentiable almost everywhere.

Let $t\geq0$ be a point at which $\bar{x}$, $\beta$, and $\eps$ are differentiable at $t$. Applying inequality \eqref{eq:solutionmap} with $t_{1}=t$ and $t_{2}=t+h$, for $h>0$, we obtain
\begin{align*}
\frac{\norm{\bar{x}(t+h)-\bar{x}(t)}}{h}&\leq\frac{\abs{\beta(t+h)-\beta(t)}}{h\eps(t)}
\norm{\opB(\bar{x}(t))}\\
&\quad+\frac{\abs{\eps(t+h)-\eps(t)}}{h\eps(t)}\norm{\bar{x}(t+h)}.
\end{align*}
Letting $h\downarrow0$ and using the continuity of $\bar{x}$, we obtain
$$
\norm{\dot{\bar{x}}(t)}\leq\frac{\abs{\dot{\beta}(t)}}{\eps(t)}\norm{\opB(\bar{x}(t))}+
\frac{\abs{\dot{\eps}(t)}}{\eps(t)}\norm{\bar{x}(t)}\qquad\text{a.e. }t\geq0.
$$
Since $\beta$ is non-decreasing, $\dot{\beta}(t)\geq0$ almost
everywhere. Therefore, we have \eqref{eq:boundCP}.
\end{proof}
}
\section{Penalty-regulated dynamical systems for constrained variational inequalities}
\label{sec:dynamics}
Recall that a function $f:[0,b]\to\scrH$ (where $b>0$) is said to be absolutely continuous if there exists an integrable function $g:[0,b]\to\scrH$ such that 
\[
f(t)=f(0)+\int_{0}^{t}g(s)\dif s\quad\forall t\in[0,b].
\]
\begin{definition}\label{def:exists}
Let $f:\R_{\geq 0}\times\scrH\to\scrH$ be a vector field depending on time and space, and let $x_{0}\in \scrH$ be given. We say $x:\R_{\geq0}\to\scrH$ is a strong global solution of 
\begin{equation}\label{eq:ODE}\tag{D}
\left\{\begin{array}{ll} 
\dot{x}(t)=f(t,x(t))\\ 
x(0)=x_{0}\in\scrH
\end{array}\right.
\end{equation}
\begin{itemize}
\item [(i)] $x:\R_{\geq0}\to\scrH$ is absolutely continuous on each interval $[0,b],0<b<\infty$; 
\item[(ii)]  $\dot{x}(t)=f(t,x(t))$ for almost every $t\in\R_{>0}$.
\end{itemize}
\end{definition} 
Existence and strong uniqueness of non-autonomous systems can be proven by means of the classical Cauchy-Lipschitz Theorem (see e.g. \cite[][Theorem 54]{sontag2013mathematical}). To use this, we need to ensure the following properties enjoyed by the vector field $f(t,x)$.
\begin{theorem}\label{th:existence}
Let $f:\R_{\geq0}\times\scrH\to\scrH$ be a given function satisfying:
\begin{itemize}
\item[(f1)] $f(\cdot,x): \R_{\geq0}\to\scrH$ is measurable for each $x\in\scrH$;
\item[(f2)] $f(t,\cdot):\scrH\to\scrH$ is continuous for each $t\geq 0$;
\item[(f3)] there exists a function $\ell(\bullet)\in L^{1}_{\text{loc}}(\R_{\geq 0};\R)$ such that 
\begin{equation}
\norm{f(t,x)-f(t,y)}\leq\ell(t)\norm{x-y}\qquad\forall t\in [0, b], \forall b\in \R_{\geq 0}, \forall x,y\in\scrH;
\end{equation}
\item[(f4)] for each $x\in\scrH$ there exists a function $\Delta(\bullet)\in L^{1}_{\text{loc}}(\R_{\geq 0};\R)$ such that 
\begin{equation}
\norm{f(t,x)}\leq\Delta(t)\qquad\forall t\in [0, b], \forall b\in \R_{\geq 0}. 
\end{equation}
\end{itemize}
Then, the dynamical system \eqref{eq:ODE} admits a unique strong solution $t\mapsto x(t)$, $t\geq 0$. 
\end{theorem}
\subsection{Penalty-regulated forward-backward dynamics}
\label{sec:FB}
In this section we study explicitly the case where the involved single-valued operators $\opD$ and $\opB$ are \emph{both} cocoercive. 
\begin{assumption}\label{ass:cocoercive}
The operator $\opD:\scrH\to\scrH$ is $\eta$-cocoercive for some $\eta>0$.
 \end{assumption}
\revise{Let $\gamma:\R_{\geq0}\to\R_{>0}$ be continuous.} Under this additional assumption on the data, we investigate the asymptotic properties of the following evolution equation
\begin{equation}\label{eq:VecFB}
\dot{x}(t)=\gamma(t)\left(\res_{\lambda(t)\opA}\big(x(t)-\lambda(t)V_{t}(x(t))\big)-x(t)\right),\; x(0)=x_{0}\in\scrH
\end{equation}

\begin{remark} \label{rem:discrete}
The dynamical system \eqref{eq:VecFB} contains a small generalization compared to \eqref{eq:FB} by including the time-scale parameter $t\mapsto \gamma(t)$. In a discrete-time formulation of the dynamical system, the scaling function $t\mapsto \gamma(t)$ can be interpreted as a relaxation parameter of the resulting numerical scheme. 
Let $\pi_{\delta}=\{0=0<t_{1}<t_{2}<\ldots<t_{n^{\delta}}=T\}$ be a partition of $[0,T]$; let $h^{\delta}_{k}\eqdef t_{k+1}^{\delta}-t_{k}^{\delta}$, and let $\delta$ be the mesh size given by $\delta\eqdef\max\{h_{k}^{\delta}:0\leq k\leq n^{\delta}-1\}$. We define a sequence $(X^{\delta}_{k})_{k\in\{0,1,\ldots,n^{\delta}-1\}}$ recursively by 
\begin{align*}
X^{\delta}_{k+1}&=(1-\gamma^{\delta}_{k}h^{\delta}_{k})X^{\delta}_{k}+h^{\delta}_{k}\gamma^{\delta}_{k}\res_{\lambda_{k}^{\delta}\opA}\big(X^{\delta}_{k}-\lambda^{\delta}_{k}V_{k}(X^{\delta}_{k})\big),\\
X_{0}^{\delta}&=x_{0}\in\scrH.
\end{align*}
This is a relaxed forward-backward splitting method, involving the iteration-dependent mapping $V_{k}(x)=\opD(x)+\eps_{k} x+\beta_{k}\opB(x)$. A version of this scheme, without relaxation and in presence of only a penalty term, has been studied in \cite{Bot:2014aa}. 
\end{remark}

\subsubsection{Existence and uniqueness of strong solutions}

In this section we verify that the data defining the evolution equation \eqref{eq:VecFB} satisfy all conditions stated in Theorem \ref{th:existence}. Whence, existence and uniqueness of strong solutions for \eqref{eq:VecFB} follow.

\begin{proposition}\label{prop:FBexistence}
 Consider the dynamical system \eqref{eq:VecFB}, where the parameter function $\gamma, \lambda:\R_{\geq 0}\to \R_{>0}$ is continuous, and Assumption \ref{ass:cocoercive} being in place. Then, for every $t\geq0$ and every $x,y \in \scrH$ we have
\begin{align}\label{eq:FBLipschitz}
&\norm{f(t,x)-f(t,y)}\leq\gamma(t)(2+\lambda(t)L(t))\norm{x-y},\text{ and }\\
&(\forall x \in \scrH)(\forall b>0),\quad f(\,\cdot\,,x)\in L^1([0,b],\scrH).\label{eq:FBIntegrable}
\end{align}
\end{proposition}
\begin{proof}
Properties $(f1),(f2)$ are clearly satisfied. To simplify the verification of the remaining properties, we set $\res_t\equiv\res_{\lambda(t)\opA}$ and $R_t(x)\eqdef x-\lambda(t)V_{t}(x)$, so that $T_{t}\eqdef \res_{t}\circ R_{t}$. The evolution equation \eqref{eq:VecFB} admits then the simpler representation $f(t,x)=\gamma(t)(T_{t}(x)-x)$. Since the resolvent $\res_t$ is firmly non-expansive, it follows that
\begin{align*}
\norm{f(t,x)-f(t,y)}&=\gamma(t)\norm{T_t(x)-x-T_t(y)+y}\\
&\leq \gamma(t)\big(\norm{T_t(x)-T_t(y)}+\norm{x-y}\big)\\
&=\gamma(t)\big(\norm{\res_{t}\circ R_{t}(x)-\res_{t}\circ R_{t}(y)}+\norm{x-y}\big)\\
&\leq \gamma(t)\big(\norm{R_{t}(x)- R_{t}(y)}+\norm{x-y}\big)\\
&\leq \gamma(t)\big(\norm{x-\lambda(t)V_t(x)-y+\lambda(t)V_t(y)}+\norm{x-y}\big)\\
&\leq\gamma(t)(2+\lambda(t)L(t))\norm{x-y}.
\end{align*}
where $L(t)=\frac{1}{\eta}+\eps(t)+\frac{\beta(t)}{\mu}.$ 

As $\lambda,\eps,\beta:\R_{\geq0} \to\R_{>0}$ are continuous on each interval $[0,b]$, where $0<b<\infty$, we get
$$
L_f:\R_{\geq 0}\rightarrow \R_{\geq0},\quad L_f(t)=\gamma(t)(2+\lambda(t)L(t)),
$$
which is clearly a locally \revise{integrable} function. This verifies condition $(f3)$.
It remains to establish condition $(f4)$. From the continuity of $\lambda,\eps,\beta$, there exist $\lambda_{\min},\eps_{\min},\beta_{\min}$, such that
$$
0<\lambda_{\min}<\lambda(t),\quad 0<\eps_{\min}<\eps(t)\text{ and } 0<\beta_{\min}<\beta(t),\quad\forall t\in [0,b].
$$
Hence, we have for all $t\in [0,b]$, using the triangle inequality, nonexpansiveness of $\res_t$ and eq. \eqref{eq:resolventContinuous}, we obtain
\begin{align*}
\norm{f(t,x)}&\leq \gamma(t)\norm{T_t(x)-x}\leq \gamma(t)\big(\norm{T_{t}(x)}+\norm{x}\big)\\
&\leq\gamma(t)\left(\norm{x}+\norm{\res_{\lambda(t)\opA}(x-\lambda_{\min}V_{\eps_{\min},\beta_{\min}}(x))}\right)\\
&+\gamma(t)\norm{\res_{\lambda(t)\opA}(x-\lambda(t)V_{\eps(t),\beta(t)}(x))-\res_{\lambda(t)\opA}(x-\lambda_{\min}V_{\eps_{\min},\beta_{\min}}(x))}\\
&\leq \gamma(t)\norm{x}+\gamma(t)\norm{\res_{\lambda(t)\opA}(x-\lambda_{\min}V_{\eps_{\min},\beta_{\min}}(x))}\\
&+\gamma(t)\|x-\lambda(t)V_{\eps(t),\beta(t)}(x)-x+\lambda_{\min}V_{\eps_{\min},\beta_{\min}}(x)\|\\
&\leq  \gamma(t)\norm{x}+\gamma(t)\norm{\res_{\lambda_{\min}\opA}(x-\lambda_{\min}V_{\eps_{\min},\beta_{\min}}(x))}\\
&+\gamma(t)(\lambda(t)-\lambda_{\min})\norm{\opA_{\lambda_{\min}}(x-\lambda_{\min}V_{\eps_{\min},\beta_{\min}}(x))}\\
&+\gamma(t)(\lambda(t)-\lambda_{\min})\norm{\opD(x)}+\gamma(t)(\lambda(t)\eps(t)-\lambda_{\min}\eps_{\min})\norm{x}\\
&+\gamma(t)(\lambda(t)\beta(t)-\lambda_{\min}\beta_{\min})\norm{\opB(x)}.
\end{align*}
Property $(f4)$ follows by integrating.
\end{proof}

\subsubsection{Strong convergence of the trajectories to the least-norm solution of \eqref{eq:MI}}

Our asymptotic analysis of the FB-dynamical system \eqref{eq:VecFB} relies on Lyapunov techniques, building on the following technical result. 
\begin{lemma} \label{Lemma:FB}
Let Assumptions \ref{ass:standing} and \ref{ass:cocoercive} be in place. Furthermore, we impose the conditions 
\begin{align}\label{eq:lambdaFB}
\lambda(t)<\frac{\eta}{1+\eta\eps(t)}\;\text{ and } \lambda(t)<\frac{\mu}{\mu\eps(t)+\beta(t)}\qquad\forall t\geq 0.
\end{align}
Then, we have 
\begin{equation*}
2\inner{\dot{x}(t),x(t)-\bar{x}(t)}\leq \gamma(t)\lambda(t)\eps(t)(\lambda(t)\eps(t)-2)\norm{x(t)-\bar{x}(t)}^{2}.
\end{equation*}
\end{lemma}
\begin{proof}
Recall, that \eqref{eq:VecFB} can be written as $\dot{x}(t)=\gamma(t)(T_{t}(x)-x)$, with $T_{t}(x)=x\iff x\in\zer(\Phi_{t})$. With this remark in mind, and using the definition of the dynamics, we observe that 
\begin{align*}
2&\inner{\dot{x}(t),x(t)-\bar{x}(t)}=\norm{\dot{x}(t)+x(t)-\bar{x}(t)}^{2}-\norm{\dot{x}(t)}^{2}-\norm{x(t)-\bar{x}(t)}^{2}\\
&=\norm{\gamma(t)(T_{t}(x(t))-\bar{x}(t))+(1-\gamma(t))(x(t)-\bar{x}(t))}^{2}-\norm{\dot{x}(t)}^{2}-\norm{x(t)-\bar{x}(t)}^{2}\\
&=\gamma(t)\norm{T_{t}(x(t))-\bar{x}(t)}^{2}+(1-\gamma(t))\norm{x(t)-\bar{x}(t)}^{2}-\gamma(t)(1-\gamma(t))\norm{T_{t}(x(t))-x(t)}^{2}\\
&\; -\norm{\dot{x}(t)}^{2}-\norm{x(t)-\bar{x}(t)}^{2}
\end{align*}
where the last equality uses \eqref{eq:convexHilbert}. On the other hand, 
\begin{align*}
&\norm{(\Id_{\scrH}-\lambda(t)V_{t})(x)-(\Id_{\scrH}-\lambda(t)V_{t})(y)}^{2}\\
&=\norm{(1-\lambda(t)\eps(t))(x-y)-\lambda(t)\left(\opD(x)-\opD(y)+\beta(t)(\opB(x)-\opB(y))\right)}^{2}\\
&=(1-\lambda(t)\eps(t))^{2}\norm{x-y}^{2}\\
&-2\lambda(t)(1-\lambda(t)\eps(t))\inner{x-y,\opD(x)-\opD(y)-\beta(t)(\opB(x)-\opB(y))}\\
&+\lambda^{2}(t)\norm{\opD(x)-\opD(y)+\beta(t)(\opB(x)-\opB(y)}^{2}
\end{align*}

Since $\opD$ and $\opB$ are cocoercive, we have $\inner{x-y,\opD(x)-\opD(y)}\geq \eta\norm{\opD(x)-\opD(y)}^{2}$ and $\inner{x-y,\opB(x)-\opB(y)}\geq \mu\norm{\opB(x)-\opB(y)}^{2}.$ Moreover
\[
\norm{\opD(x)-\opD(y)+\beta(t)(\opB(x)-\opB(y))}^{2}\leq 2\norm{\opD(x)-\opD(y)}^{2}+2\beta(t)^{2}\norm{\opB(x)-\opB(y)}^{2}, 
\]
so that we obtain 
\begin{align*}
&\norm{(\Id_{\scrH}-\lambda(t)V_{t})(x)-(\Id_{\scrH}-\lambda(t)V_{t})(y)}^{2}\\
&\leq (1-\lambda(t)\eps(t))^{2}\norm{x-y}^{2}+2\lambda(t)(\lambda(t)-\eta(1-\lambda(t)\eps(t)))\norm{\opD(x)-\opD(y)}^{2}\\
&+2\lambda(t)\beta(t)(\lambda(t)\beta(t)-\mu(1-\lambda(t)\eps(t)))\norm{\opB(x)-\opB(y)}^{2}.
\end{align*}
Thanks to \eqref{eq:lambdaFB}, we remain with 
\[
\norm{(\Id_{\scrH}-\lambda(t)V_{t})(x)-(\Id_{\scrH}-\lambda(t)V_{t})(y)}^{2}\leq (1-\lambda(t)\eps(t))^{2}\norm{x-y}^{2}\qquad\forall t\geq 0. 
\]
Therefore, using the non-expansiveness of the resolvent, we can continue the previous estimate to obtain
\begin{align*}
2\inner{\dot{x}(t),x(t)-\bar{x}(t)}=&\ \gamma(t)\norm{T_{t}(x(t))-T_{t}(\bar{x}(t))}^{2}+(1-\gamma(t))\norm{x(t)-\bar{x}(t)}^{2}\\
&\ -\gamma(t)(1-\gamma(t))\norm{T_{t}(x(t))-x(t)}^{2}-\norm{\dot{x}(t)}^{2}-\norm{x(t)-\bar{x}(t)}^{2}\\
\leq &\ \gamma(t)\lambda(t)\eps(t)(\lambda(t)\eps(t)-2)\norm{x(t)-\bar{x}(t)}^{2}.
\end{align*}
\end{proof}

The next assumption strengthens the condition spelled out in Lemma \ref{Lemma:FB}, and makes sure that both inequalities formulated in eq. \eqref{eq:lambdaFB} are satisfied.

\begin{assumption}\label{ass:Thm_FB}
The function $\lambda:\R_{\geq0}\to\R_{>0}$ is continuous, with
\[
\lambda(t)<\frac{1}{1/\eta+\eps(t)+\beta(t)/\mu}\qquad\forall t\geq 0.
\]
\end{assumption}

\begin{theorem}\label{th:MainFB}
Let $t\mapsto x(t)$ be the strong solution of \eqref{eq:FB}. Let Assumptions \ref{ass:standing}-\ref{ass:Thm_FB} be in place. Moreover, consider the following scaling relations:
\begin{align}
&\lim_{t\to\infty}\frac{\dot{\eps}(t)}{\gamma(t)\lambda(t)\eps^{2}(t)}=0,\label{eq:scaleTikFB}\\ 
&\lim_{t\to\infty}\frac{\dot{\beta}(t)}{\gamma(t)\lambda(t)\eps^{2}(t)}=0,\label{eq:scalePenaltyFB} \text{ and }\\ 
&\int_0^{\infty}\gamma(t)\lambda(t)\eps(t)(2-\lambda(t)\eps(t))\dif t=\infty.\label{eq:integralFB}
\end{align}
Then, $x(t)\to \Pi_{\zer(\opA+\opD+\NC_{\scrC})}(0)$ as $t\to\infty$.
\end{theorem}

\begin{proof}
Set $\theta(t)\eqdef \frac{1}{2}\norm{x(t)-\bar{x}(t)}^{2}$. We then have 
\begin{align*}
\dot{\theta}(t)&=\inner{x(t)-\bar{x}(t),\dot{x}(t)-\frac{\dif}{\dif t}\bar{x}(t)}\\
&=\inner{x(t)-\bar{x}(t),\dot{x}(t)}-\inner{x(t)-\bar{x}(t),\frac{\dif}{\dif t}\bar{x}(t)}\\
&\leq \inner{x(t)-\bar{x}(t),\dot{x}(t)}+\norm{x(t)-\bar{x}(t)}\cdot\norm{\frac{\dif}{\dif t}\bar{x}(t)}.
\end{align*}
Setting $\delta(t)\eqdef-\gamma(t)\lambda(t)\eps(t)(\frac{\lambda(t)\eps(t)}{2}-1)$ and $\Delta(t,0)\eqdef\int_{0}^{t}\delta(s)\dif s$, Lemma \ref{Lemma:FB} gives 
\[
\dot{\theta}(t)\leq -2\delta(t) \theta(t)+\norm{x(t)-\bar{x}(t)}\cdot\norm{\frac{\dif}{\dif t}\bar{x}(t)}.
\]
Following Lemma \ref{lem:CP}, we can bound the second addendum to get
\[
\dot{\theta}(t)\leq -2\delta(t)\theta(t)+\sqrt{2\theta(t)}\left(\frac{\dot{\beta}(t)}{\eps(t)}\norm{\opB(\bar{x}(t))}+\frac{\abs{\dot{\eps}(t)}}{\eps(t)}\norm{\bar{x}(t)}\right).
\]
By putting $\varphi(t)\eqdef\sqrt{2\theta(t)}$, we thus finally arrive at the inequality 
\begin{equation}
\dot{\varphi}(t)\leq -\delta(t)\varphi(t)+w(t), 
\end{equation}
where (recall that $\dot{\eps}(t)\leq 0$)
\begin{equation}
w(t)\eqdef \frac{\dot{\beta}(t)}{\eps(t)}\norm{\opB(\bar{x}(t))}-\frac{\dot{\eps}(t)}{\eps(t)}\norm{\bar{x}(t)}.
\end{equation} 
Introducing the integration factor $\exp(-\Delta(t,0))$, we thus obtain 
$$\varphi(t)\leq \varphi(0)\exp(-\Delta(t,0))+\exp(-\Delta(t,t_0))\int_{0}^{t}w(s)\exp(\Delta(s,t_0))\dif s.$$
If the integral on the right-hand side is bounded, we are done. Else, apply l'H\^{o}spital's rule, the definition of $\delta(t)$, and conditions \eqref{eq:scaleTikFB} and \eqref{eq:scalePenaltyFB}, to get 
\begin{align*}
\lim_{t\to\infty}\exp(-\Delta(t,0))\int_{0}^{t}w(s)&\exp(\Delta(s,0))\dif s=\lim_{t\to\infty}\frac{w(t)\exp(\Delta(t,0))}{\delta(t)\exp(\Delta(t,0))}\\
&=\lim_{t\to\infty}\frac{\dot{\eps}(t)}{\delta(t)\eps(t)}\left(\frac{\dot{\beta}(t)}{\dot{\eps}(t)}\norm{\opB(\bar{x}(t))}-\norm{\bar{x}(t)}\right)=0.
\end{align*}
This shows that $\lim_{t\to\infty}\norm{x(t)-\bar{x}(t)}=0$. Since 
\[
\norm{x(t)-\Pi_{\zer(\opA+\opD+\NC_{\scrC})}(0)}\leq\norm{x(t)-\bar{x}(t)}+\norm{\bar{x}(t)-\Pi_{\zer(\opA+\opD+\NC_{\scrC})}(0)}, 
\]
the strong convergence claim follows from Corollary \ref{cor:asymptotics}. 
\end{proof}

\begin{remark}
Since $\eps(t)\to 0$ and $\beta(t)\to\infty$ as $t\to\infty$, the hypothesis \eqref{eq:lambdaFB} implies that $\lim_{t\to\infty}\lambda(t)=0$, as well as $\limsup_{t\to\infty}\lambda(t)\beta(t)\leq\mu$. We also see that $\lambda(t)<\eta$ for all $t\geq 0$. Hence, the rate of the decay of the step size must be on par with the rate of divergence of the penalty parameter. 
\end{remark}

\begin{remark}
The assumptions formulated in Theorem \ref{th:MainFB} can be satisfied by the following set of functions: \revise{$\gamma(t)=1+\frac{1}{t+b}$}, $\lambda(t)=\frac{\lambda}{\beta(t)+\lambda\eps(t)}$, where $\lambda<c\min\{\mu,\eta\}$ for some $c\in(0,1)$ (recommended to be close to 1), as well as $\eps(t)=(t+b)^{-r}$ and $\beta(t)=(t+b)^{s}$ for $b\geq 1$ and $0<r<s$. Then, \revise{$\delta(t)=\scrO\big(\frac{\eps(t)}{\beta(t)}\big)=\scrO\big((t+b)^{-(r+s)}\big)$, and consequently we need to impose the restriction $s+r\leq 1$} to ensure that $\delta\notin L^{1}(\R_{+})$. Additionally, we compute $\frac{\dot{\eps}(t)}{\gamma(t)\lambda(t)\eps^2(t)}=\scrO\big((t+b)^{r+s-1}\big)$. This yields $r+s<1$. Finally, we have $\frac{\dot{\beta}(t)}{\gamma(t)\lambda(t)\eps^2(t)}=\scrO\big((t+b)^{2(r+s)-1}\big)$, and to make this a bounded sequence, we have the restriction $r+s<\frac{1}{2}$. These conditions together span a region of feasible parameters $(r,s)$ which is nonempty.
\end{remark}
\subsection{Penalty-regulated forward-backward-forward dynamics}
\label{sec:FBF}
A critical assumption underlying the forward-backward dynamical system \eqref{eq:FB} is the cocoercivity (inverse strong monotonicity) of the single-valued operators $\opD$ and $\opB$. Cocoercivity is guaranteed to hold when the monotone inclusion problem \eqref{eq:MI} models optimality conditions for constrained convex optimization problems. However, it generically fails in structured monotone splitting problems arising from primal-dual optimality conditions derived from the Fenchel-Rockafellar theorem. Section \ref{sec:Applications} describes a very general class of splittings illustrating this claim. Motivated by this observation, this section derives a new dynamical system formulation exhibiting multiscale aspects, respecting Tikhonov regularization and penalization. Specifically, the class of dynamical systems we are investigating in this section builds on \cite{Bot:2020aa}, and extends it to the constrained setting.\\ 

We consider the following first-order dynamical system 
\begin{align*}
p(t)&=\res_{\lambda(t)\opA}(x(t)-\lambda(t)V_{t}(x(t)),\\
\dot{x}(t)&=p(t)-x(t)+\lambda(t)[V_{t}(x(t))-V_{t}(p(t))]
\end{align*}
To obtain a simpler expression of this dynamics, we define the reflection $R_{t}(x)\eqdef x-\lambda(t)V_{t}(x)$, and the vector field $f(t,x):\R_{\geq0}\times\scrH\to\scrH$ given by
\begin{equation}\label{eq:Tseng}
f(t,x)\eqdef \left(R_{t}\circ \res_{\lambda(t)\opA}\circ R_{t}\right)(x)-R_{t}(x)
\end{equation}
The first-order dynamical system \eqref{eq:FBF} is then exactly of the form \eqref{eq:ODE}. To prove existence and uniqueness of strong global solutions, we can use the same arguments as in Section 5.1 of \cite{Bot:2020aa}, based on the Cauchy-Lipschitz theorem for absolutely continuous trajectories. We therefore omit these straightforward derivations. 

\begin{remark} 
A discrete-time formulation of the FBF dynamical system can be given along the same lines as in Remark \ref{rem:discrete}. Adopting the notation used there, let $\pi_{\delta}=\{0<t_{1}<t_{2}<\ldots<t_{n^{\delta}}=T\}$ be a partition of $[0,T]$; let $h^{\delta}_{k}\eqdef t_{k+1}^{\delta}-t_{k}^{\delta}$, and let $\delta$ be the mesh size given by $\delta\eqdef\max\{h_{k}^{\delta}:0\leq k\leq n^{\delta}-1\}$. We define a pair process $(P^{\delta}_{k},X^{\delta}_{k})_{k\in\{0,1,\ldots,n^{\delta}-1\}}$ recursively by 
\begin{align*}
P^{\delta}_{k}&=\res_{\lambda^{\delta}_{k}\opA}(X^{\delta}_{k}-\lambda^{\delta}_{k}V_{k}(X^{\delta}_{k})),\\
X^{\delta}_{k+1}&=P^{\delta}_{k}+\lambda^{\delta}_{k}(V_{k}(X^{\delta}_{k})-V_{k}(P^{\delta}_{k})). 
\end{align*}
This is just a dynamic version of Tseng's modified extragradient method, augmented with penalty terms (as studied in \cite{Bot:2014aa}), together with a Tikhonov term (a new ingredient).
\end{remark}

\subsubsection{Strong convergence of the trajectories to the least-norm solution}
This section contains the necessary Lyapunov analysis for understanding the long-run behavior of trajectories issued by \eqref{eq:FBF}. We begin with some technical lemmata.

\begin{lemma}
\label{Lemma 3.7} 
	For almost all $t \in\R_{\geq 0}$, we obtain
	\begin{align*}
		0 &\leq-\norm{x(t)-p(t)}^{2}+\norm{x(t)-\bar{x}(t)}^{2}-(1+2\lambda(t)\eps(t))\norm{p(t)-\bar{x}(t)}^{2}\\
		&~~~~+2\lambda(t)\Big\langle V_{t}(p(t))-V_{t}(x(t)),p(t)-\bar{x}(t)\Big\rangle
	\end{align*}
\end{lemma}
\begin{proof}
From the definition of the process $t\mapsto p(t)$, we deduce
	$$
	(\Id_{\scrH}+\lambda(t)\opA)p(t) \ni x(t)-\lambda(t)V_{t}(x(t)),
	$$
	it follows,
	$$
	\Phi_{t}(p(t))=\opA(p(t))+V(t,p(t)) \ni \frac{x(t)-p(t)}{\lambda(t)}-V_{t}(x(t))+V_{t}(p(t))=-\frac{\dot{x}(t)}{\lambda(t)}.
	$$
Recall that $\Phi_t\equiv \Phi_{\varepsilon(t),\beta(t)}$ is $\eps(t)$-strongly monotone, so that $\{\bar{x}(t)\}=\zer(\Phi_t)$. Consequently,
	\begin{equation}\label{eq:Phi-Monotone}
	\Big\langle-\frac{\dot{x}(t)}{\lambda(t)}-0,p(t)-\bar{x}(t)\Big\rangle\geq \varepsilon(t)\norm{p(t)-\bar{x}(t)}^2.
	\end{equation}
Using these properties, we obtain
\begin{align}\label{e3.5}
	2\lambda(t)\varepsilon(t)\norm{p(t)-\bar{x}(t)}^2&\leq 2\inner{ x(t)-p(t),p(t)-\bar{x}(t)}\\
	&+2\lambda(t)\Big\langle V_t(p(t))-V_t(x(t)),p(t)-\bar{x}(t)\Big\rangle \nonumber \\
	&=-\norm{x(t)-p(t)}^{2}+\norm{x(t)-\bar{x}(t)}^{2}-\norm{p(t)-\bar{x}(t)}^{2} \nonumber \\
	&+2\lambda(t)\Big\langle V_t(p(t))-V_t(x(t)),p(t)-\bar{x}(t)\Big\rangle, \nonumber
\end{align}
which completes the proof.
\end{proof}

\begin{lemma}\label{Lemma 3.8} 
	Let $t \mapsto x(t)$ be the strong global solution of \eqref{eq:FBF}. Then
	$$
	\inner{x(t)-\bar{x}(t),\dot{x}(t)} \leq (\lambda(t)L(t)-1)\norm{x(t)-p(t)}^2-\lambda(t)\eps(t)\norm{p(t)-\bar{x}(t)}^2
	$$
	for almost all $t\geq 0$.
\end{lemma}
\begin{proof}
	For almost all $t\geq 0$, we have
	\begin{align*}
		2\langle x(t)-\bar{x}(t),\dot{x}(t) \rangle &= 2\langle x(t)-\bar{x}(t),p(t)-x(t) \rangle \\
		&+2\Big\langle x(t)-\bar{x}(t),\lambda(t)[V_{t}(x(t))-V_{t}(p(t))] \Big\rangle\\
		&=\|\bar{x}(t)-p(t)\|^2-\|\bar{x}(t)-x(t)\|^2-\|x(t)-p(t)\|^2\\
		&+2\lambda(t)\Big\langle x(t)-\bar{x}(t),V_{t}(x(t))-V_{t}(p(t)) \Big\rangle
	\end{align*}
	combining with Lemma \ref{Lemma 3.7}, we get
		\begin{align*}
			\|\bar{x}(t)-p(t)\|^2-\|\bar{x}(t)-x(t)\|^2
			&\leq-\|x(t)-p(t)\|^2-2\lambda(t)\varepsilon(t)\|p(t)-\bar{x}(t)\|^2\\
			&~~~~+2\lambda(t)\Big\langle V_{t}(p(t))-V_{t}(x(t)),p(t)-\bar{x}(t) \Big\rangle
		\end{align*}
	then,
	\begin{align*}
		2\langle x(t)-\bar{x}(t),\dot{x}(t) \rangle &\leq -2\|x(t)-p(t)\|^2-2\lambda(t)\varepsilon(t)\|p(t)-\bar{x}(t)\|^2\\
		&+2\lambda(t)\Big\langle V_{t}(p(t))-V_{t}(x(t)),p(t)-\bar{x}(t) \Big\rangle\\
		&+2\lambda(t)\Big\langle x(t)-\bar{x}(t),V_{t}(x(t))-V_{t}(p(t)) \Big\rangle\\
		&= -2\|x(t)-p(t)\|^2-2\lambda(t)\varepsilon(t)\|p(t)-\bar{x}(t)\|^2\\
		&+2\lambda(t)\Big\langle x(t)-p(t), V_{t}(x(t))-V_{t}(p(t)) \Big\rangle\\
		&\leq -2\|x(t)-p(t)\|^2-2\lambda(t)\varepsilon(t)\|p(t)-\bar{x}(t)\|^2\\
		&+2\lambda(t)\|x(t)-p(t)\|\cdot\| V_{t}(x(t))-V_{t}(p(t))\|\\
		&\leq -2\|x(t)-p(t)\|^2-2\lambda(t)\varepsilon(t)\|p(t)-\bar{x}(t)\|^2\\
		&+2\lambda(t)L(t)\|x(t)-p(t)\|^2\\
		&\leq -2(1-\lambda(t)L(t))\|x(t)-p(t)\|^2-2\lambda(t)\varepsilon(t)\|p(t)-\bar{x}(t)\|^2
  	\end{align*}
	the proof is completed.
\end{proof}


\begin{theorem}\label{th:main}
	Let $t \mapsto x(t)$ be the strong global solution of \eqref{eq:FBF}. Let Assumptions \ref{ass:standing}-\ref{ass:Thm_FB} be in place. Furthermore, we impose the following conditions:
	\begin{itemize}
		\item[(i)] $\lim\limits_{t \rightarrow \infty}\int_{0}^{t}\delta(s)\dif s=\infty$, where $\delta(t)=\frac{1-\lambda(t)L(t)}{a^{2}(t)}$, and 
\begin{equation}\label{eq:a}
a(t)\eqdef 2+\frac{1}{\lambda(t)\varepsilon(t)}+\frac{1}{\eta\varepsilon(t)}+\frac{\beta(t)}{\mu\varepsilon(t)}.
\end{equation}
		 \item[(ii)] $\lim_{t\to\infty}\frac{\dot{\eps}(t)}{\eps(t)\delta(t)}=0$ and $\lim_{t\to\infty}\frac{\dot{\beta}(t)}{\eps(t)\delta(t)}=0$.
	\end{itemize}
Then $x(t)\rightarrow \Pi_{\opA+\opD+\NC_{\scrC}}(0)$ as $t \rightarrow +\infty$.
\end{theorem}
\begin{proof}
	Define $\theta(t)=\frac{1}{2}\|x(t)-\bar{x}(t)\|^2$ where $t \geq 0$. 
    \revise{
    Since both $x$ and $\bar{x}$ are locally absolutely continuous, we have for almost every $t\geq0$,
 $$
 \dot{\theta}(t)=\inner{x(t)-\bar{x}(t),\dot{x}(t)-\dot{\bar{x}}(t)},
 $$
 }
combining \eqref{eq:boundCP} with Lemma \ref{Lemma 3.8}, we get
\revise{
	\begin{equation}\label {e3.6}
	\begin{aligned}
\dot{\theta}(t)&=\Big\langle x(t)-\bar{x}(t),\dot{x}(t)-\frac{\dd}{\dd t}\bar{x}(t)\Big\rangle\\
&\leq-(1-\lambda(t)L(t))\norm{x(t)-p(t)}^{2}-\lambda(t)\eps(t)\norm{p(t)-\bar{x}(t)}^{2}\\
&\quad+\norm{x(t)-\bar{x}(t)}\norm{\dot{\bar{x}}(t)}\\
&\leq-(1-\lambda(t)L(t))\norm{x(t)-p(t)}^{2}-\lambda(t)\eps(t)\norm{p(t)-\bar{x}(t)}^{2}\\
&\quad+\frac{\abs{\dot{\eps}(t)}}{\eps(t)}\norm{x(t)-\bar{x}(t)}\norm{\bar{x}(t)}+\frac{\dot{\beta}(t)}{\eps(t)}\norm{x(t)-\bar{x}(t)}\norm{\opB(\bar{x}(t))}.
	\end{aligned}
\end{equation}
}
%
%
Since $\Phi_{t}=\opA+V_{t}$ is $\varepsilon(t)$-strongly monotone, \eqref{eq:Phi-Monotone} shows that 
\[
	\lambda(t)\varepsilon(t)\|p(t)-\bar{x}(t)\|^2 \leq \langle x(t)-p(t)+\lambda(t)[V_{t}(p(t))-V_{t}(x(t))],p(t)-\bar{x}(t) \rangle,
	\]
Using Cauchy-Schwarz, and the $L(t)$-Lipschitz continuity of $V_{t}$, we obtain 
\[
\norm{p(t)-\bar{x}(t)}\leq \left(\frac{1}{\lambda(t)\varepsilon(t)}+1+\frac{1}{\eta\varepsilon(t)}+\frac{\beta(t)}{\mu\varepsilon(t)}\right)\|x(t)-p(t)\|.
 \]
It follows that
	\begin{align*}
\|x(t)-\bar{x}(t)\|&\leq \|x(t)-p(t)\|+\|p(t)-\bar{x}(t)\|\\
&\leq\big(2+\frac{1}{\lambda(t)\varepsilon(t)}+\frac{1}{\eta\varepsilon(t)}+\frac{\beta(t)}{\mu\varepsilon(t)}\big)\|x(t)-p(t)\|=a(t)\|x(t)-p(t)\|.
\end{align*}
For almost all $t \geq 0$, we thus get 
\begin{equation}\label {e3.9}
	-\|x(t)-p(t)\|^2 \leq -\frac{1}{a^2(t)}\|x(t)-\bar{x}(t)\|^2.
\end{equation}
Defining $\varphi \eqdef\sqrt{2\theta}$, we obtain
\revise{
\begin{align*}
\dot{\theta}(t)&=\dot{\varphi}(t)\varphi(t)\leq-\frac{1-\lambda(t)L(t)}{a^{2}(t)}\varphi^{2}(t)+\frac{\abs{\dot{\eps}(t)}}{\eps(t)}\varphi(t)\norm{\bar{x}(t)}+\frac{\dot{\beta}(t)}{\eps(t)}\varphi(t)\norm{\opB(\bar{x}(t))}\\
&=-\frac{1-\lambda(t)L(t)}{a^{2}(t)}\varphi^{2}(t)-\frac{\dot{\eps}(t)}{\eps(t)}\varphi(t)\norm{\bar{x}(t)}+\frac{\dot{\beta}(t)}{\eps(t)}\varphi(t)\norm{\opB(\bar{x}(t))}.
\end{align*}
}
We define $\delta(t)\eqdef\frac{1-\lambda(t)L(t)}{a^2(t)}$, and the integrating factor $\Delta(t)\eqdef\int_{0}^{t}\delta(s)\dd s$. Upon using the simplified notation $\bar{x}(t)\equiv\bar{x}(\eps(t),\beta(t))$, we then continue from the previous display with 
\revise{
$$
\frac{\dd}{\dd t}\big(\varphi(t)\exp(\Delta(t))\big)\leq\exp(\Delta(t))\left[-\frac{\dot{\eps}(t)}{\eps(t)}\norm{\bar{x}(t)}+\frac{\dot{\beta}(t)}{\eps(t)}\norm{\opB(\bar{x}(t))}\right].
$$
}
\revise{We set $w(t)\eqdef-\frac{\dot{\eps}(t)}{\eps(t)}\norm{\bar{x}(t)}+\frac{\dot{\beta}(t)}{\eps(t)}\norm{\opB(\bar{x}(t))}$, Since $\dot{\eps}(t)\leq0$ and $\dot{\beta}(t)\geq0$ almost everywhere, we have $w(t)\geq0$.} And subsequently integrate both sides in the previous display from 0 to $t$, to conclude with
\revise{
\begin{equation}\label{e4.12}
0\leq\varphi(t)\leq\exp(-\Delta(t))\left[\varphi(0)+\int_0^t\exp(\Delta(s))w(s)\dd s\right].
\end{equation}
}
If \revise{$t\mapsto \int_{0}^{t}\exp(\Delta(s))w(s)\dd s$} happens to be bounded, then we immediately obtain from hypothesis (i) that $\varphi(t)\to 0$. Otherwise, we apply l'H\^{o}pital's rule to get 
\begin{align*}
\lim_{t\to\infty}\exp(-\Delta(t))\int_{0}^{t}\exp(\Delta(s))w(s)\dd s=\lim_{t\to\infty}\frac{\exp(\Delta(t))w(t)}{\delta(t)\exp(\Delta(t))}=\lim_{t\to\infty}\frac{w(t)}{\delta(t)}.
\end{align*}
Additionally, we know from the proof of Proposition \ref{prop:asymptotics} that both $t\mapsto\norm{\opB(\bar{x}(t))}$ and $t\mapsto\norm{\bar{x}(t)}$ are bounded. \revise{Moreover, $\frac{w(t)}{\delta(t)}=-\frac{\dot{\eps}(t)}{\eps(t)\delta(t)}\norm{\bar{x}(t)}+
\frac{\dot{\beta}(t)}{\eps(t)\delta(t)}\norm{\opB(\bar{x}(t))}.$ Hence, by condition (ii), $\lim_{t\to\infty}\frac{w(t)}{\delta(t)}=0$. Consequently, $\varphi(t)\to 0$ and therefore $\norm{x(t)-\bar{x}(t)}\to 0$.}
By the triangle inequality $\norm{x(t)-\Pi_{\zer(\Phi)}(0)} \leq \norm{x(t)-\bar{x}(t)}+\norm{\bar{x}(t)-\Pi_{\zer(\Phi)}(0)}$. Using Corollary \ref{cor:asymptotics}, we conclude $x(t)\rightarrow \Pi_{\zer(\Phi)}(0)$ as $t \rightarrow +\infty$.
\end{proof}
\begin{remark}
We give some concrete specifications for functions $\eps(t),\lambda(t)$ and $\beta(t)$ satisfying all conditions for Theorem \ref{th:main} to hold. Writing Assumption \ref{ass:Thm_FB} in dynamical terms, we obtain the condition $\lambda(t)L(t)<1$. We claim that
$$
\liminf_{t\to\infty}(1-\lambda(t)L(t))=1-\limsup_{t\to\infty}\lambda(t)L(t)>0.
$$
Indeed, using the definition of the Lipschitz constant $L(t)$ in \eqref{eq:Lt}, we obtain 
$$
\lambda(t)L(t)=(1/\eta+\eps(t))\lambda(t)+\lambda(t)\beta(t)/\mu, 
$$
so that 
$\limsup_{t\to\infty}\lambda(t)L(t)<1.$ Additionally, 
\begin{align*}
a(t)&=2+\frac{1}{\eps(t)}\left(\frac{1}{\lambda(t)}+\frac{1}{\eta}+\frac{\beta(t)}{\mu}\right)=\frac{\lambda(t)(\eps(t)+L(t))+1}{\lambda(t)\eps(t)}\\
&=\frac{L(t)}{\eps(t)}(1+\frac{\eps(t)}{L(t)}+\frac{1}{L(t)\lambda(t)})=\scrO(\beta(t)/\eps(t))
\end{align*}
using that $L(t)=\scrO(\beta(t))$. This in turn implies $\delta(t)=\frac{1-\lambda(t)L(t)}{a^{2}(t)}=\scrO(\frac{\eps^{2}(t)}{\beta^{2}(t)})$. Hence, $\lim_{t\to\infty}\delta(t)=0$, and for obtaining $\delta\notin L^{1}(\R_{\geq 0})                 $ it suffices to guarantee that $\int_{0}^{\infty}\frac{\eps^{2}(t)}{\beta^{2}(t)}\dd t=\infty$. Then, 
$$
\frac{\dot{\eps}(t)}{\eps(t)\delta(t)}=\frac{\dot{\eps}(t)L^{2}(t)}{\eps^{3}(t)}\frac{(1+\frac{\eps(t)}{L(t)}+\frac{1}{L(t)\lambda(t)})^{2}}{1-\lambda(t)L(t)}\\
=\frac{\dot{\eps}(t)\beta^{2}(t)}{\eps^{3}(t)}\scrO(1).$$
It therefore suffices to have $\lim\limits_{t\to\infty}\frac{\dot{\eps}(t)\beta^{2}(t)}{\eps^{3}(t)}=0$. By a similar argument, it is easy to see that $\frac{\dot{\beta}(t)}{\eps(t)\delta(t)}=\frac{\dot{\beta}(t)\beta(t)^{2}}{\eps(t)^{3}}\scrO(1)$. Therefore, it suffices to ensure that $\frac{\dot{\beta}(t)\beta(t)^{2}}{\eps(t)^{3}}$ is bounded. Finding such functions is not too difficult. 

Assume $\eps(t)=(t+b)^{-r}, \beta(t)=(t+b)^s$, where $s,b> 0$ and $r$ is chosen such that $r+s>0$ and $0<r<s$. Then $\frac{\eps^{2}(t)}{\beta^{2}(t)}=(t+b)^{-2(r+s)}$, and consequently we need to impose the restriction $s+r<\frac{1}{2}$ to ensure that $\delta\notin L^{1}(\R_{+})$. Additionally, we compute $\frac{\dot{\eps}(t)\beta(t)^{2}}{\eps(t)^{3}}=-r(t+b)^{2(r+s)-1}$. This yields the same restriction $r+s<\frac{1}{2}$. Finally, $\frac{\dot{\beta}(t)\beta(t)^{2}}{\eps^{3}(t)}=s(t+b)^{3(r+s)-1}$, and to make this a bounded sequence, we need to impose the condition $s+r<\frac{1}{3}$. These conditions together span a region of feasible parameters $(r,s)$ which is nonempty. 
\close
\end{remark}

\section{Multiscale forward-backward dynamics}
\label{sec:Multiscale}
In this section we extend the forward-backward penalty dynamics to the challenging case in which the feasible set $\scrC$ can be represented as the set of joint minimizers of two penalty terms. This structural assumption can be motivated by decomposition methods in optimization and optimal control, in which different structural properties of the feasible set can be incorporated via different penalty functions. To make this concrete, we refine the general problem formulation \eqref{eq:MI} by assuming that the constrained domain $\scrC$ is of the form 
$$
\scrC=\scrC_{1}\cap \scrC_{2}, 
$$ 
with $\scrC_{1},\scrC_{2}$ closed convex and nonempty sets. This geometry could represent conic domains intersected with an affine subspace, leading to a very generic class of variational problems. To be consistent with our penalty formulation, we impose the following structural assumption on the geometry.
\begin{assumption}\label{ass:Multi}
The constrained domain $\scrC\subset\scrH$ admits the representation
\begin{equation}\label{eq:C-Multi}
\scrC=\argmin\Psi_{1}\cap\argmin\Psi_{2},
\end{equation}
with convex potentials $\Psi_{1},\Psi_{2}:\scrH\to(-\infty,\infty)$ satisfying 
\begin{enumerate}
\item $\Psi_{1}:\scrH\to\R$ is convex and $L_{\Psi_{1}}$-smooth; 
\item $\Psi_{2}:\scrH\to \R\cup\{+\infty\}$ is proper, lower semicontinuous and convex with subdifferential $\partial\Psi_{2}$;
\item $\Psi_{1}+\Psi_{2}:\scrH\to\R\cup\{+\infty\}$ is coercive;
\item The operator $\opA+\partial\Psi_{2}:\scrH\to 2^{\scrH}$ is maximally monotone.
\end{enumerate}
\end{assumption}

To align the notation with the previously studied penalty dynamics, we set $\opB_{1}\eqdef \nabla \Psi_{1}$ and $\opB_{2}\eqdef\partial\Psi_{2}$. Note that $\opB_{1}$ is $\frac{1}{L_{\Psi_{1}}}$-co-coercive and monotone, and $\opB_{2}$ is maximally monotone. Since our penalization framework only uses information on the gradient and subgradients of the penalty potentials, we can assume without loss of generality that $\argmin \Psi_{i}=\Psi^{-1}_{i}(0)$ for $i\in\{1,2\}$. If this is not originally the case, we can always re-shift the graph of the function so that the problem formulation remains the same. To solve the constrained variational inequality \eqref{eq:MI} within this more structured setup, we propose a forward-backward based dynamical system via a full splitting of the resulting problem. Given positive functions $\lambda(t),\beta(t),\eps(t)$, we assume that the time-varying operator 
\begin{equation}
\opA_{t}(x)\eqdef \opA(x)+\beta(t)\opB_{2}(x)
\end{equation}
has an easy-to-compute resolvent mapping $\res_{\lambda(t)\opA_{t}}=(\Id_{\scrH}+\lambda(t)\opA_{t})^{-1}$. This evaluation condition is the main reason why we assume that $\opA+\partial\Psi_{2}$ is maximally monotone. We further assume that the resolvent is everywhere single-valued and nonexpansive. Proceeding then in the spirit of the dynamical system \eqref{eq:FB}, we design a vector field which exploits a full splitting of the operators involved by moving all single-valued operators into the backward step and all set-valued information into the forward step. Hence, we arrive at the following first-order dynamical system 
\begin{equation}\tag{SFBP}\label{eq:SFBP-dynamics}
\dot{x}(t)+x(t)=\res_{\lambda(t)\opA_{t}}(x(t)-\lambda(t)V_{t}(x(t)))
\end{equation}
where 
\begin{equation}
V_{t}(x)\eqdef\opD(x)+\eps(t)x+\beta(t)\opB_{1}(x). 
\end{equation}

With the introduction of the time-varying operators $\opA_{t}$ and $V_{t}$, we achieve a full splitting of the penalized auxiliary problems of the form \eqref{eq:MIauxiliary}. This is done on purpose to reduce computational costs in the implementation of the dynamics. To the best of our knowledge, the first full splitting dynamics of this kind has been studied in \cite{czarnecki2016splitting} in the potential case and without Tikhonov regularization. We extend their analysis to the monotone operator case and add Tikhonov regularization on top of the operators to induce strong convergence. Moreover, the dynamical system \eqref{eq:SFBP-dynamics} contains the penalty-regulated forward-backward dynamical system \eqref{eq:FB} by setting $\opB_{2}=0$.  

We shall prove the convergence of the trajectories associated with \eqref{eq:SFBP-dynamics} under the Attouch-Czarnecki condition \cite{Attouch:2010aa,AttCzar18}:
\begin{assumption}[Attouch-Czarnecki condition]
\begin{equation}\label{eq:AttCza}
(\forall \xi\in\range(\NC_{\scrC})):\; \int_{0}^{\infty}\lambda(t)\beta(t)\left[(\Psi_{1}+\Psi_{2})^{\ast}(\frac{\xi}{\beta(t)})-\sigma_{\scrC}(\frac{\xi}{\beta(t)})\right]\dif t<\infty.
\end{equation}
\end{assumption}
This condition involves the Fenchel conjugate of the convex function $\Psi_{1}+\Psi_{2}$, defined as  
$$
(\Psi_{1}+\Psi_{2})^{\ast}(y)\eqdef \sup_{x\in\scrH}\{\inner{y,x}-(\Psi_{1}(x)+\Psi_{2}(x))\}
$$
The integrability criterion \eqref{eq:AttCza} has been exploited heavily in the analysis of penalty regulated dynamical systems. Early references include \cite{AttCzarPey11,Bot:2014aa,Bot:2016aa,Peypouquet:2012aa}, among many others. \revise{As discussed in Remark \ref{rem:errorbound}, it is implied by an Hölderian growth condition of the combined penalty function $\Psi=\Psi_{1}+\Psi_{2}$. In particular, if $\Psi$ satisfies a Hölderian growth condition with exponent $\rho\in(1,2]$ around $\argmin\Psi=\scrC$, then
$$
0\leq \Psi^{\ast}(z)-\sigma_{\setC}(z)\leq\tau^{1-\rho^{*}}\frac{1}{\rho^{*}}\norm{z}^{\rho^{*}}\qquad \forall z\in\R^{d},
$$
where $\rho^{*}$ is the conjugate exponent. Consequently, \eqref{eq:AttCza} is satisfied whenever
$$
0\leq \Psi^{\ast}(z)-\sigma_{\setC}(z)\leq\tau^{1-\rho^{*}}\frac{1}{\rho^{*}}\norm{z}^{\rho^{*}}\qquad \forall z\in\R^{d},
$$
}


In the important case where $\Psi(z)=\frac{1}{2}\dist_{\scrC}(z)^{2}$, then $\Psi^{\ast}(z)-\sigma_{\scrC}(z)=\frac{1}{2}\norm{z}^{2}$, and thus the integrability condition \eqref{eq:AttCza} simplifies to $\int_{0}^{\infty}\frac{\lambda(t)}{\beta(t)}\dif t<\infty$.

\subsection*{General weak convergence analysis}

The asymptotic behavior of the multi-penalty system turned out to be more complex. However, we can establish the weak ergodic convergence of the trajectories, following the arguments in \cite{Attouch:2010aa,AttCzarPey11}. The proof of the following result can be found in Appendix \ref{sec:Appendix_Multi}.

\begin{theorem}\label{th:MS1}
Let $\eps,\lambda:\R_{\geq 0}\to\R_{> 0}$ be absolutely continuous functions in $L^{2}(\R_{> 0})\setminus L^{1}(\R_{> 0})$, and such that $\lim_{t\to\infty}\eps(t)=\lim_{t\to\infty}\lambda(t)=0$ and \revise{$\lim_{t\to\infty}\frac{\lambda(t)}{\eps(t)}=0$}. Suppose, moreover, that $\liminf_{t\to\infty}\lambda(t)\beta(t)>0$ and, for every $\xi\in\range(\NC_{\scrC})$, we have
$$
\int_{0}^{\infty}\lambda(t)\beta(t)\left[(\Psi_{1}+\Psi_{2})^{\ast}(\frac{\xi}{\beta(t)})-\sigma_{\scrC}(\frac{\xi}{\beta(t)})\right]\dif t<\infty.
$$    
Then, 
$$
\lim_{t\to\infty}\norm{\opB_{1}(x(t))}=\lim_{t\to\infty}(\Psi_{1}+\Psi_{2})(x(t)+\dot{x}(t))=0.
$$
Define the ergodic trajectory
$$
\bar{x}(T)\eqdef \frac{\int_{0}^{T}\lambda(s)x(s)\dif s}{\int_{0}^{T}\lambda(s)\dif s} \quad (T>0).
$$
Then, \revise{$\bar x(T)$} converges weakly, as \revise{$T\to\infty$}, to a point in $\zer(\opA+\opD+\NC_{\scrC})$.
\end{theorem}


\revise{
\begin{remark}
Assume that the function $\Psi\eqdef\Psi_{1}+\Psi_{2}$ is boundedly inf-compact, which means that every set of the form
\[\{x\in\scrH\vert\; \norm{x}\leq R\wedge(\Psi_{1}+\Psi_{2})(x)\leq M\},\]
with $R\geq 0$ and $M\in\R$, is relatively compact. Under the assumptions of Theorem \ref{th:MS1}, the weak convergence of the ergodic trajectory $\bar{x}(T)$ can be strengthened to strong convergence. 
\close
\end{remark}

\begin{proof}
Define the ergodic trajectory
$$
\bar{z}(T)\eqdef\frac{\int_{T_{0}}^{T}\lambda(s)(x(s)+\dot{x}(s))\dif s}{\int_{T_{0}}^{T}\lambda(s)\dif s},
$$
Since $\lambda \in L^{2}(\R_{> 0})\setminus L^{1}(\R_{> 0})$ and $\int_{T_0}^{\infty}\norm{\dot{x}(t)}^{2}\dif t<\infty$, by the Cauchy--Schwarz inequality, we have
\begin{align*}
\norm{\bar{z}(T)-\bar{x}(T)}
&=\frac{1}{\Lambda(T_0,T)}\norm{\int_{T_0}^{T}\lambda(s)\dot{x}(s)\dif s}\\
&\leq\frac{1}{\Lambda(T_0,T)}\int_{T_0}^{T}\lambda(s)\norm{\dot{x}(s)}\dif s\\
&\leq\frac{\left(\int_{T_0}^{T}\lambda^{2}(s)\dif s\right)^{1/2} \left(\int_{T_0}^{T}\norm{\dot{x}(s)}^{2}\dif s\right)^{1/2}}{\Lambda(T_0,T)}\to0
\end{align*}

Moreover, by Jensen's inequality,
\begin{align*}
\Psi(\bar{z}(T))
&=\Psi\left(\frac{1}{\Lambda(T_0,T)}\int_{T_0}^{T}\lambda(s)(x(s)+\dot{x}(s))\dif s
\right)\\
&\leq\frac{1}{\Lambda(T_0,T)}\int_{T_0}^{T}\lambda(s)\Psi(x(s)+\dot{x}(s))\dif s\to0,
\end{align*}
where the limit follows from $\Psi(x(t)+\dot{x}(t))\to0$ and $\int_{T_0}^{T}\lambda(s)\dif s\to\infty$. By Theorem \ref{th:MS1},
$\bar{x}(T)\rightharpoonup\bar{x}_{\infty}$ for some $\bar{x}_{\infty}\in\zer(\opA+\opD+\NC_{\scrC})$. Hence $\bar{z}(T)\rightharpoonup\bar{x}_{\infty}$ and
$(\bar{z}(T))_{T\geq T_{0}}$ is bounded. Since $\Psi(\bar{z}(T))\to0$, the bounded inf-compactness of $\Psi$ implies that $\bar{z}(T)$ is asymptotically relatively compact. Since every sequence $T_n\to\infty$ admits a strongly convergent subsequence and $\bar{x}_{\infty}$ is the unique weak limit, we obtain $\bar{z}(T)\to\bar{x}_{\infty}$ strongly as $T\to\infty$. Finally, since $\norm{\bar{z}(T)-\bar{x}(T)}\to0$, we conclude that $\bar{x}(T)\to\bar{x}_{\infty}$ strongly as $T\to\infty$.
\end{proof}

}
\section{Applications}
\label{sec:Applications}
In this section we describe some prototypical applications of our splitting framework.
\subsection{Sparse optimal control of linear systems}
\revise{Given $y_{0}\in\R^{n}$, a matrix $A\in\R^{n\times n}$, a matrix-valued function $B:[0,T]\to\R^{n\times m}$, and a vector-valued function $c:[0,T]\to\R^{n}$}, consider the control system 
\begin{equation}\label{eq:CS}
\dot{y}(t)=A y(t)+B(t)u(t)+c(t)\quad y(0)=y_{0}.
\end{equation}
The process $u\in L^{\infty}(0,T;\R^{m})$ is an open-loop control. We assume that $B$ and $c$ are bounded and sufficiently regular, so that the $u\in L^{\infty}(0,T;\R^{m})$, and the system has an absolutely continuous solution, denoted by $Y^{u,y_{0}}_{\bullet}:[0,T]\to\R^{n}$. We are interested in solving the optimal control problem 
\begin{equation}
\min\{\frac{1}{2}\norm{Y^{u,y_{0}}_{\bullet}-\bar{y}(\bullet)}_{L^{2}(0,T;\R^{n})}^{2}+\alpha_{1}\norm{u}^{2}_{L^{2}(0,T;\R^{m})}+\alpha_{2}\norm{u}_{L^{1}(0,T;\R^{m})}\}
\end{equation}
where the minimum is taken over the set of admissible controls 
$$
\scrU:=\{u\in L^{\infty}(0,T;\R^{m})\vert u(\bullet)\text{ is measurable and }\norm{u(t)}_{\infty}\leq 1~~ \text{a.e. }t\in[0,T]\}
$$
Let $P:[0,T]\to\R^{n\times n}$ denote the resolvent of the matrix equation $\dot{X}=AX$, $X(0)=\Id_{\R^{n}}$ satisfying $P(t)=\exp(tA)$. Then 
$$
Y^{y_{0},u}_{t}=P(t)y_{0}+P(t)\int_{0}^{t}P(s)^{-1}[B(s)u(s)+c(s)]\dif s 
$$ 
satisfies 
\revise{
$$
\frac{\dif}{\dif t}Y^{y_{0},u}_{t}=AY^{y_{0},u}_{t}+B(t)u(t)+c(t),\quad Y^{y_{0},u}_{0}=y_{0}.
$$
}
This in turn is equivalent to 
$$
S(u,y)+z_{0}=0,
$$ 
where 
\begin{align*}
&S(u,y)(t)\eqdef -y(t)+P(t)\int_{0} ^{t}P(s)^{-1}B(s)u(s)\dif s ,\text{ and }\\
&z_{0}(t)\eqdef P(t)y_{0}+P(t)\int_{0}^{t}P(s)^{-1}c(s)\dif s 
\end{align*}
Set $\scrH\eqdef L^{2}(0,T;\R^{m})\times L^{2}(0,T;\R^{n})$. $S$ is a bounded linear operator from $\scrH$ to $L^{2}(0,T;\R^{n})$, and consequently the function $\Psi_{1}:\scrH\to\R$ defined by 
$$
\Psi_{1}(u,y)\eqdef \frac{1}{2}\norm{S(u,y)+z_{0}}^{2}_{L^{2}(0,T;\R^{n})} 
$$
is convex and continuously differentiable. Next, define $\Psi_{2}(u,y)=\delta_{\scrU}(u)$ to obtain a convex, proper and lower semi-continuous function. Moreover, $(u,y)\in\scrH$ solves the control system if and only if $(u,y)\in\argmin(\Psi_{1}+\Psi_{2})$. With this notation, 
\begin{align*}
J_{1}(u,y)&\eqdef \frac{1}{2}\norm{y-\bar{y}}^{2}_{L^{2}(0,T;\R^{n})}+\alpha_{1}\norm{u}^{2}_{L^{2}(0,T;\R^{m})},\\
J_{2}(u,y)&\eqdef \alpha_{2}\norm{u}_{L^{1}(0,T;\R^{m})} 
\end{align*}
so that the optimal control problem becomes 
$$
\min\{J_{1}(u,y)+J_{2}(u,y):(u,y)\in\argmin(\Psi_{1}+\Psi_{2})\} 
$$
With $\opD\eqdef \nabla J_{1}$ and $\opA\eqdef \partial J_{2}$, so that $\scrC=\argmin(\Psi_{1}+\Psi_{2})$, we arrive at a constrained variational inequality problem of the form \eqref{eq:MI}. 

\subsection{Monotone inclusions involving compositions with linear continuous operators}
We next show how our method can be applied to solve monotone inclusion problems involving compositions of operators, as proposed by \cite{combettes2012primal,Bot:2014ab}
Let $\scrH$ and $\scrG$ be real Hilbert spaces. We introduce operators $\opA_{1}:\scrH\to 2^{\scrH}$ and $\opA_{2}:\scrG\to 2^{\scrG}$ which we assume to be maximally monotone. Additionally, we let $L:\scrH\to\scrG$ represent a linear continuous operator. Lastly, we consider $\opD:\scrH\to\scrH$ monotone and  and $\frac{1}{\eta}$-Lipschitz continuous operator with $\eta>0$, and a monotone operator $\opB:\scrH\to\scrH$ a monotone and $\frac{1}{\mu}$-Lipschitz continuous operator with $\mu>0$ satisfying $\scrC=\zer(\opB)\neq\emptyset$. The monotone inclusion problem to solve is
\begin{equation}\label{eq:Variational}
0\in \opA_{1}(x)+L^{\ast}\circ \opA_{2}\circ Lx+\opD(x)+\NC_{\scrC}(x).
\end{equation}
This splitting gains relevance in the generic convex optimization model 
$$
\min_{x\in\scrC}\{f(x)+h(x)+g(Lx)\}
$$
where $f\in\Gamma_{0}(\scrH),g\in\Gamma_{0}(\scrG)$, $K:\scrH\to\scrG$ is a bounded linear operator and $h\in\bC^{1,1}_{L_{h}}(\scrH;\R)$. \revise{Assume that the standard qualification conditions ensuring the validity of the Fenchel--Rockafellar duality theorem and the Moreau--Rockafellar subdifferential sum rule are satisfied.} Via the classical Fenchel-Rockafellar duality, we can transform this problem into the constrained saddle point problem 
$$
\min_{x\in\scrC}\max_{y}\{f(x)+h(x)+\inner{L x,y}-g^{\ast}(y)\}
$$
which amounts to solving the monotone inclusion problem consisting in finding a pair $(x^{\ast},y^{\ast})$ such that 
\begin{align*}
&0\in \partial f(x^{\ast})+\nabla h(x^{\ast})+L^{\ast}y+\NC_{\scrC}(x^{\ast})\\
&0\in Lx^{\ast}-\partial g^{\ast}(y^{\ast})
\end{align*}
 Since $\partial g^{\ast}=(\partial g)^{-1}$, we can combine these two inclusions to a single one reading as 
 $$
 0\in\partial f(x^{\ast})+\nabla h(x^{\ast})+L^{\ast}\partial g (L x^{\ast})+\NC_{\scrC}(x^{\ast}).
 $$
 We thus arrive at an instantiation of problem \eqref{eq:Variational}, by identifying $\opA_{1}=\partial f,\opA_{2}=\partial g,\opD=\nabla h$.\\

We use the product space approach in order to show that the general problem \eqref{eq:Variational} can be formulated as the monotone inclusion problem \eqref{eq:MI}. To this end, we consider the product space $\scrH\times\scrG$ endowed with the inner product 
 $$
 \inner{(x,y),(x',y')}_{\scrH\times\scrG}=\inner{x,x'}_{\scrH}+\inner{y,y'}_{\scrG} 
 $$
 and corresponding norm. We define the operators 
 $$
 \tilde{\opA}(x,y)\eqdef \opA_{1}(x)\times \opA_{2}^{-1}(y),\; \tilde{\opD}(x,y)\eqdef \begin{pmatrix} \opD(x)+L^{\ast}y\\ -L x\end{pmatrix},\;\tilde{\opB}(x,y)\eqdef\begin{pmatrix} \opB(x)\\0\end{pmatrix},
 $$
 and for $\tilde{\scrC}\eqdef \scrC\times\scrG=\zer(\tilde{\opB})$, 
 $$
 \NC_{\tilde{\scrC}}(x,v)= \NC_{\scrC}(x)\times\{0\}.
 $$
 One can easily show that if $(x,v)\in\zer(\tilde{\opA}+\tilde{\opD}+\NC_{\tilde{\scrC}})$, then $x\in\zer(\opA_{1}+L^{\ast}\opA_{2}L+\opD+\NC_{\scrC})$.  Conversely, when $x\in\zer(\opA_{1}+L^{\ast}\opA_{2}L+\opD+\NC_{\scrC})$, then there exists $v\in \opA_{2}(Lx)$ such that $(x,v)\in\zer(\tilde{\opA}+\tilde{\opD}+\NC_{\tilde{\scrC}})$. Thus, determining the zeros of operator $\tilde{\opA}+\tilde{\opD}+\NC_{\tilde{\scrC}}$ will provide a solution for the monotone inclusion problem \eqref{eq:Variational}. 
 
 $\tilde{\opA}$ is maximally monotone \cite[Proposition 20.23]{BauCom16}, $\tilde{\opD}$ is monotone and $\tilde{\eta}$-Lipschitz continuous, where $\tilde{\eta}=\sqrt{2(1/\eta^{2}+\norm{K}^{2})}$, and $\tilde{B}$ is monotone and $(1/\mu)$-Lipschitz continuous. We can thus directly use our dynamical system to determine zeros of $\tilde{A}+\tilde{D}+\NC_{\tilde{\scrC}}$. We write the trajectory in terms of pairs $t\mapsto (p(t),q(t))$ and $t\mapsto (x(t),y(t))$ given by 
 \begin{align*}
& p(t)=\res_{\lambda(t)\opA_{1}}(x(t)-\lambda(t)(\opD(x(t))+\eps(t)x(t)+L^{\ast}y(t)+\beta(t)\opB(x(t)))\\ 
& q(t)=\res_{\lambda(t)\opA_{2}^{-1}}(y(t)+\lambda(t)Lx(t)-\lambda(t)\eps(t)y(t))\\
& \dot{x}(t)=(1-\lambda(t)\eps(t))(p(t)-x(t))\\
&\quad +\lambda(t)[\opD(x(t))-\opD(p(t))+\beta(t)(\opB(x(t))-\opB(p(t)))+L^{\ast}(y(t)-q(t))]\\
& \dot{y}(t)=(1-\lambda(t)\eps(t))(q(t)-y(t))+\lambda(t)L(p(t)-x(t)).
\end{align*} 

\begin{remark}
Let us underline the fact that, even in the situation when $B$ is cocoercive and, hence, $\tilde{B}$ is cocoercive, the forward-backward penalty scheme studied in \cite{Bot:2016aa} cannot be applied in this context, because the operator $\tilde{D}$ is definitely not cocoercive. This is due to the presence of the skew operator $(x,y)\mapsto (K^{\ast}y,Kx)$ in its definition. This fact provides a good motivation for formulating, 
\revise{a forward-backward-forward penalty scheme for the monotone inclusion problem investigated in this paper. Related Tseng-type penalty schemes have also been investigated in \cite{boct2014vietnam}.}

\end{remark}

\subsubsection{Application to linear inverse problems}
Building on the primal-dual splitting approach of \cite{Bot:2014ab,Bot:2016ab}, we consider a linear inverse problem with forward operator $K:\Rn\to\R^{m}$ which is the problem of finding $\theta\in\R^{n}$ that solves the linear system 
\[
K\theta=b
\]
Typically, this linear system is ill-posed, and therefore a regularization framework is adopted. A popular formulation is to consider flattened gradient via an isotropic total variation regularization, which reads as the simple bilevel optimization problem 
\begin{equation}\label{eq:TV}
\min_{\theta\in[0,1]^{n}}\textsf{TV}(\theta) \quad\text{s.t.: }\theta\in S:=\argmin_{\theta'\in\Rn}\{\frac{1}{2}\norm{K\theta'-b}^{2}\}
\end{equation}
where the mapping $\textsf{TV}:\Rn\to\R$ is defined as 
\begin{align*}
\textsf{TV}(\theta)=&\sum_{i=1}^{M-1}\sum_{j=1}^{N-1}\sqrt{(\theta_{i+1,j}-\theta_{i,j})^{2}+(\theta_{i,j+1}-\theta_{i,j})^{2}}+\sum_{i=1}^{M-1}\abs{\theta_{i+1,N}-\theta_{i,N}}\\
&+\sum_{j=1}^{N-1}\abs{\theta_{M,j+1}-\theta_{M,j}}
\end{align*}
and $\theta_{i,j}$ denotes the normalized value of the pixel located in the $i$-th row and the $j$-th column, for $i\in\{1,\ldots,M\}$ and $j\in\{1,\ldots,N\}$. Let $\scrH=\Rn$ and $\scrY=\R^{n}\times\R^{m}$. Define the linear operator $L:\Rn\to\scrY$ by $\theta\mapsto (L_{1}\theta,L_{2}\theta)\in\scrY$ defined coordinate-wise by
\begin{align*}
L_{1}\theta_{i,j}=\left\{\begin{array}{ll} 
\theta_{i+1,j}-\theta_{i,j} & \text{ if }i<M,\\
0 & \text{else.}
\end{array}\right.,\quad 
L_{2}\theta_{i,j}=\left\{\begin{array}{ll} 
\theta_{i,j+1}-\theta_{i,j} & \text{ if }i<N,\\
0 & \text{else.}
\end{array}\right.
\end{align*}
 $L$ represents a discretization of the gradient using Neumann boundary conditions. We note that $\norm{L}^{2}\leq 8$. 
 
 For $(y,z),(u,v)\in\scrY$, we introduce the inner product 
 $$
 \inner{(y,z),(u,v)}:=\sum_{i=1}^{M}\sum_{j=1}^{N}(y_{i,j}u_{i,j}+z_{i,j}v_{i,j}),
 $$
 with the corresponding norm $\norm{(y,z)}_{\scrY}=\sqrt{\inner{(y,z),(y,z)}}$. It then follows $\mathsf{TV}(\theta)=\norm{L\theta}_{\scrY}$. The dual norm to $\norm{\cdot}_{\scrY}$ is defined as 
 $$
 \norm{(u,v)}_{\scrY,\ast}:=\sup_{\norm{(u,v)}_{\scrY}\leq 1}\inner{(y,z),(u,v)}.
 $$
Accordingly, we define the dual space $\scrY^{\ast}$, as the Euclidean space $\scrY$ endowed with the norm $\norm{\cdot}_{\scrY}$. 

Define the function $f(\theta)=\delta_{[0,1]^{n}}(\theta)$ and $g(u,v)=\norm{(u,v)}_{\scrY}$. It follows that \eqref{eq:TV} is representable as 
 $$
 \min_{\theta\in\scrH}\{f(\theta)+g(L\theta)\} \quad\text{s.t.: }\theta\in S:=\argmin_{\theta'\in\scrH}\{\frac{1}{2}\norm{K \theta'-b}^{2}\}.
 $$
 The Fenchel-Rockafellar dual approach gives us the saddle point bilevel problem
\begin{equation}\label{eq:saddle}
 \min_{\theta\in\scrH}\max_{(u,v)\in\scrY^{\ast}}\{f(\theta)+\inner{L\theta,(u,v)}-g^{\ast}(u,v)\} \quad\text{s.t.: }\theta\in S:=\argmin_{\theta'\in\scrH}\{\frac{1}{2}\norm{K \theta'-b}^{2}\}.
\end{equation}
where 
$$
g^{\ast}(u,v) =\delta_{M}(u,v), \;M=\{(u,v)\in\scrY\vert\; \norm{(u,v)}_{\scrY,\ast}\leq 1\}.
$$
This yields the optimality conditions 
\begin{align*}
&0\in\partial f(\bar{\theta})+L^{\ast}(\bar{u},\bar{v}) +\NC_{S}(\bar{\theta})\\ 
&0\in \partial g^{\ast}(\bar{u},\bar{v})-L\bar{\theta} 
\end{align*}
Define $\scrC=S\times\scrY$, so that $\NC_{\scrC}(\theta,u,v)=\NC_{S}(\theta)\times \{\0_{\scrY}\}$, to obtain the monotone inclusion 
$$
0\in \opA(\bar{\theta},\bar{u},\bar{v})+\opD(\bar{\theta},\bar{u},\bar{v})+\NC_{\scrC}(\theta,u,v).
$$
where $\opD$ is the skew symmetric linear operator $\opD(\theta,u,v)=[L^{\ast}(u,v),-L\theta]$. To solve this problem with our penalty regularized dynamical system, we relax the variational problem to arrive the unconstrained min-max optimization formulation 
\begin{equation}
\min_{\theta\in\scrH}\max_{y\in\scrY}\{f(
\theta)+\inner{L\theta,(u,v)}-g^{\ast}(u,v)+\frac{\eps}{2}\norm{\theta}^{2}-\frac{\eps}{2}\norm{(u,v)}_{\scrY}^{2}+\beta\Psi(\theta,u,v)\}, 
\end{equation}
where $\Psi(\theta,u,v)\eqdef \frac{1}{2}\norm{K\theta-b}^{2}$. Define the monotone and co-coercive operator 
$$
\opB(\theta,u,v)\eqdef \nabla\Psi(\theta,u,v) =[K^{\ast}(K\theta-b);\0_{\scrY}]\in\scrH\times\scrY,
$$
so that $\scrC=\zer(B)$. We thus can approach the solution of our linear inverse problem with the outer penalization scheme using the monotone operator 
$$
\Phi_{\eps,\beta}(\theta,u,v)=\opA(\theta,u,v)+\opD(\theta,u,v)+\eps[\theta,u,v]+\beta\opB(\theta,u,v). 
$$
For the implementation of the algorithm, we use the formulas 
$$
\res_{\lambda\partial f}=\Pi_{[0,1]^{n}},\;\res_{\lambda\partial g^{\ast}}=\Pi_{M},
$$
where $\Pi_{S}:\scrY\to M$ is defined componentwise as \cite{Bot:2014ab}
$$
(u_{i,j},v_{i,j})\mapsto\frac{(p_{i,j},q_{i,j})}{\max\{1,\sqrt{p^{2}_{i,j}+q^{2}_{i,j}}\}}\quad\forall 1\leq i\leq M,1\leq j\leq N.
$$
Writing out the iterations of the FBF penalty system, we construct two absolutely continuous functions $p(t)=[\tilde{\theta}(t),\tilde{u}(t),\tilde{v}(t)]$ and $x(t)=[\theta(t),u(t),v(t)]$ solving the following system of ODEs
\begin{align*}
&\tilde{\theta}(t)=\Pi_{[0,1]^{n}}[\theta(t)-\lambda(t)L^{\ast}(u(t),v(t))-\lambda(t)\eps(t)\theta(t)-\lambda(t)\beta(t)K^{\ast}(K\theta(t)-b)],\\
&\begin{pmatrix}
\tilde{u}(t)\\ \tilde{v}(t)\end{pmatrix}=\Pi_{M}\left[\begin{pmatrix} u(t) \\ v(t)\end{pmatrix}+\lambda(t)\begin{pmatrix} L_{1}\theta(t)\\ L_{2}\theta(t)\end{pmatrix} -\lambda(t)\eps(t)\begin{pmatrix} u(t)\\ v(t)\end{pmatrix}\right],\\
&\dot{\theta}(t)+\theta(t)=\tilde{\theta}(t)+\lambda(t)\biggl[L^{*}(u(t)-\tilde{u}(t),v(t)-\tilde{v}(t))\\
&\quad+\eps(t)(\theta(t)-\tilde{\theta}(t))+\beta(t)K^{\ast}K(\theta(t)-\tilde{\theta}(t))\biggr],\\
&\begin{pmatrix} \dot{u}(t) \\ \dot{v}(t)\end{pmatrix}+\begin{pmatrix} u(t) \\ v(t)\end{pmatrix}=\begin{pmatrix}\tilde{u}(t) \\ \tilde{v}(t)\end{pmatrix}-\lambda(t)\begin{pmatrix} L_{1}(\theta(t)-\tilde{\theta}(t)) \\ L_{2}(\theta(t)-\tilde{\theta}(t))\end{pmatrix}+\lambda(t)\eps(t)
\begin{pmatrix} u(t)-\tilde{u}(t)\\ v(t)-\tilde{v}(t)\end{pmatrix}.
\end{align*}
For the numerical experiments, we considered two different test images, discretised on a grid of size $256\times 256$, and constructed a blurred and noisy image by making first use of a Gaussian blur operator of size $9\times 9$ and standard deviation 4. Afterwards, we've added a zero mean white Gaussian noise with standard deviation $10^{-3}$. The obtained numerical results are illustrated in Figure \ref{fig:Clock} and  \ref{fig:Cameraman}.

\begin{figure}[H]
\centering
    \includegraphics[width=0.7\textwidth]{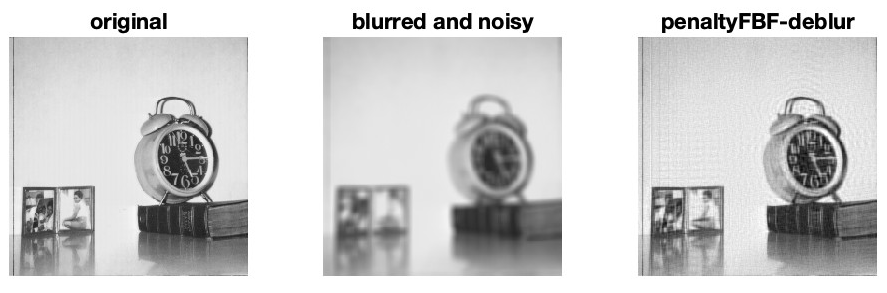}\\
     \includegraphics[width=0.4\textwidth]{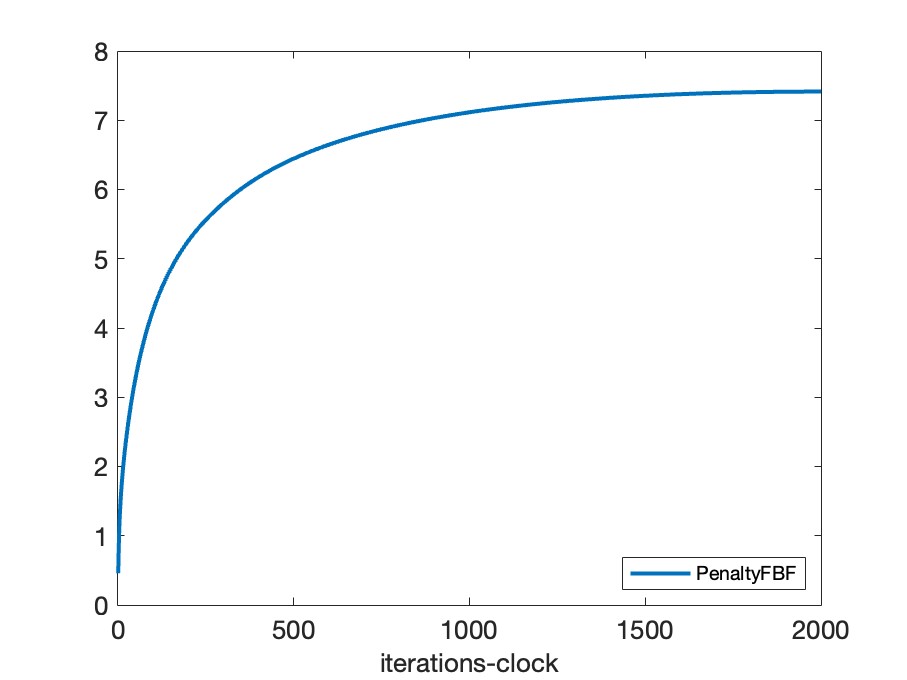}
   \caption{The figure shows that original image, the blurred image, and the reconstructed image, as well as the evolution of the ISNR for the penalty-regulated FBF dynamical system \eqref{eq:FBF}.}
   \label{fig:Clock}
\end{figure}
\begin{figure}[H]
\centering
    \includegraphics[width=0.7\textwidth]{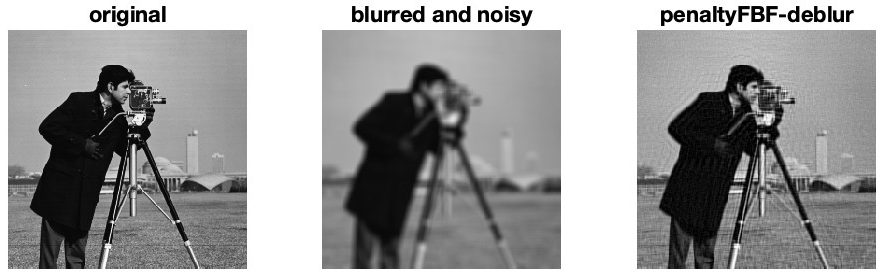}\\
    \includegraphics[width=0.4\textwidth]{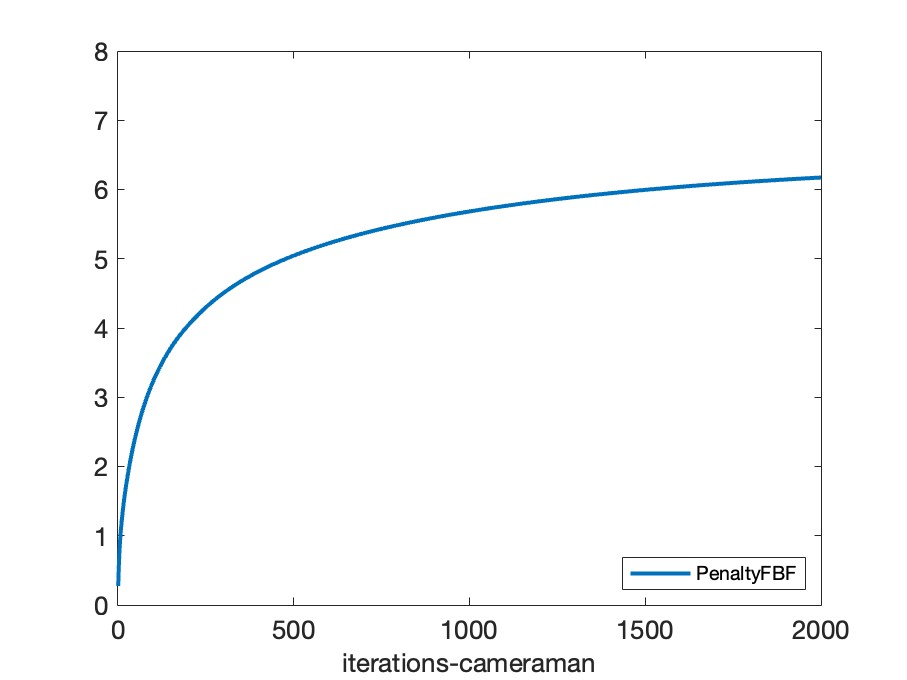}
    \caption{The figure shows that original image, the blurred image, and the reconstructed image, as well as the evolution of the ISNR for the penalty-regulated FBF dynamical system \eqref{eq:FBF}.}
    \label{fig:Cameraman}
\end{figure}
\section{Conclusion}
\label{sec:conclusion}
The asymptotic analysis of dynamical systems derived from operator splitting problems has been a very productive line of research over the past 20 years. In this paper we develop a family of dynamical systems featuring Tikhonov regularization and penalty effects. Tikhonov regularization induces strong convergence towards the minimum norm solution, whereas the penalty term steers the system to satisfy constraints subjected to the variational problem we aim to solve. We prove the asymptotic convergence in three paradigmatic settings: (i) Operator splitting with cocoercive data, (ii) Operator splitting with monotone and Lipschitz data, and (iii) Operator splitting with multiple penalty terms. Future directions of research include the extension to stochastic operator equations, such as \cite{rodrigo2024stochastic}. Other potentially interesting directions would be the inclusion of inertia or momentum effects into the dynamical system in second order investigate the potential for acceleration. We leave these interesting directions for future research.

\backmatter

\section*{Funding} This work was supported by a CSC scholarship and the FMJH Program Gaspard Monge for optimization and operations research and their interactions with data science. MST acknowledges financial support from the Deutsche Forschungsgemeinschaft (DFG) - Projektnummer 556222748 ("Non-stationary hierarchical optimization"). 

\section*{Declarations}


\begin{itemize}
\item The authors have no relevant financial or non-financial interests to disclose.
\item The authors contributed equally to this work. The first draft of this manuscript was written by Siqi Qu, and the co-authors Mathias Staudigl and Juan Peypouquet finalized the draft. 
\end{itemize}

\begin{appendices}

\section{Proof of Theorem \ref{th:MS1}}
\label{sec:Appendix_Multi}

\textbf{Step 1: Energy estimate.}
Pick $(u,w)\in\gr(\opA+\opD+\NC_{\scrC})$ so that $w=a+\opD(u)+\xi$, for $a\in\opA(u)$ and $\xi\in\NC_{\scrC}(u)$.  From the definition of the forward-backward dynamical system, we have 
$$
-\frac{1}{\lambda(t)}\dot{x}(t)-V_{t}(x(t))-\beta(t)b_{2}(t)\in\opA(x(t)+\dot{x}(t)) 
$$
for some $b_{2}(t)\in\opB_{2}(x(t)+\dot{x}(t))=\partial\Psi_{2}(x(t)+\dot{x}(t))$. Combined with $a\in\opA(u)$, we obtain
$$
\inner{a+\frac{1}{\lambda(t)}\dot{x}(t)+V_{t}(x(t))+\beta(t)b_{2}(t),u-x(t)-\dot{x}(t)}\geq 0
$$
Rearranging this, we arrive at 
\begin{align*}
\inner{\dot{x}(t),x(t)+\dot{x}(t)-u}&\leq\lambda(t)\inner{a+\opD(x(t))+\eps(t)x(t)+\beta(t)\opB_{1}(x(t)),u-x(t)-\dot{x}(t)}\\
&+\lambda(t)\beta(t)\inner{b_{2}(t),u-\dot{x}(t)-x(t)}
\end{align*}
 
The subgradient inequality yields 
\begin{align*}
0 & =\Psi_{2}(u)\geq\Psi_{2}(x(t)+\dot{x}(t))+\inner{b_{2}(t),u-x(t)-\dot{x}(t)}\\
\iff & -\Psi_{2}(x(t)+\dot{x}(t))\geq\inner{b_{2}(t),u-x(t)-\dot{x}(t)}
\end{align*}
Hence, we can continue the previous display as 
\begin{align*}
\inner{\dot{x}(t),x(t)+\dot{x}(t)-u}&\leq \lambda(t)\inner{a+\opD(x(t))+\beta(t)\opB_{1}(x(t))+\eps(t)x(t),u-\dot{x}(t)-x(t)}\\
&-\lambda(t)\beta(t)\Psi_{2}(x(t)+\dot{x}(t)). 
\end{align*}
Therefore, 
\begin{equation}\label{eq:Lyap1}
\begin{split}
\frac{\dif}{\dif t}\norm{x(t)-u}^{2}&=2\inner{\dot{x}(t),x(t)-u}\\
&=2\inner{\dot{x}(t),x(t)+\dot{x}(t)-u}-2\norm{\dot{x}(t)}^{2}\\
&\leq 2\lambda(t)\inner{a+\opD(x(t))+\eps(t)x(t),u-\dot{x}(t)-x(t)}\\
&-2\lambda(t)\beta(t)\Psi_{2}(x(t)+\dot{x}(t))-2\norm{\dot{x}(t)}^{2}\\
&+2\lambda(t)\beta(t)\inner{\opB_{1}(x(t)),u-\dot{x}(t)-x(t)}
\end{split}
\end{equation}
Since $\opB_{1}$ is $\frac{1}{L_{\Psi_{1}}}$-cocoercive, we have 
\begin{equation}
\inner{\opB_{1}(x(t)),x(t)-u}\geq\frac{1}{L_{\Psi_{1}}}\norm{\opB_{1}(x(t))}^{2}\qquad\forall u\in\scrC.
\end{equation}
Additionally, the convex gradient inequality yields 
\begin{equation}
0=\Psi_{1}(u)\geq\Psi_{1}(x(t))+\inner{\opB_{1}(x(t)),u-x(t)}
\end{equation}
Performing a convex combination of these two inequalities, we obtain for all $c_{1}>0$,
\begin{equation}\label{eq:B11-main}
\inner{\opB_{1}(x(t)),x(t)-u}\geq \frac{1}{(1+c_{1})L_{\Psi_{1}}}\norm{\opB_{1}(x(t))}^{2}+\frac{c_{1}}{1+c_{1}}\Psi_{1}(x(t)).
\end{equation}
Next, take $c_{0}>0$, and observe 
\begin{align*}
0&\leq \frac{1}{1+c_{0}}\norm{\dot{x}(t)+(1+c_{0})\lambda(t)\beta(t)\opB_{1}(x(t))}^{2}\\
&=\frac{1}{1+c_{0}}\norm{\dot{x}(t)}^{2}+(1+c_{0})\lambda(t)^{2}\beta(t)^{2}\norm{\opB_{1}(x(t))}^{2}+2\lambda(t)\beta(t)\inner{\dot{x}(t),\opB_{1}(x(t))}.
\end{align*}
Hence, 
\begin{equation}\label{eq:B12}
2\lambda(t)\beta(t)\inner{\dot{x}(t),\opB_{1}(x(t))}\geq -\frac{1}{1+c_{0}}\norm{\dot{x}(t)}^{2}-(1+c_{0})\lambda(t)^{2}\beta(t)^{2}\norm{\opB_{1}(x(t))}^{2}.
\end{equation}
On the other hand, the descent lemma for functions with a Lipschitz continuous gradient yields
$$
\Psi_{1}(x(t)+\dot{x}(t))\leq\Psi_{1}(x(t))+\inner{\opB_{1}(x(t)),\dot{x}(t)}+\frac{L_{\Psi_{1}}}{2}\norm{\dot{x}(t)}^{2},
$$
so that 
\begin{align}\label{eq:B13}
2\lambda(t)\beta(t)\inner{\opB_{1}(x(t)),\dot{x}(t)}&\geq 2\lambda(t)\beta(t)[\Psi_{1}(x(t)+\dot{x}(t))-\Psi_{1}(x(t))]\\
&-L_{\Psi_{1}}\lambda(t)\beta(t)\norm{\dot{x}(t)}^{2}\nonumber
\end{align}
A convex combination of \eqref{eq:B12} with \eqref{eq:B13} shows that 
\begin{equation}\label{eq:B12-main}
\begin{split}
2\lambda(t)\beta(t)\inner{\opB_{1}(x(t)),\dot{x}(t)}&\geq-\frac{1+c_{0}}{1+c_1}\lambda(t)^{2}\beta(t)^{2}\norm{\opB_{1}(x(t))}^{2}\\
&-\left(\frac{1}{(1+c_{1})(1+c_{0})}+\frac{c_1 L_{\Psi_{1}}\lambda(t)\beta(t)}{1+c_{1}}\right)\norm{\dot{x}(t)}^{2}\\
&+ \frac{2c_{1}\lambda(t)\beta(t)}{1+c_1}\left(\Psi_{1}(x(t)+\dot{x}(t))-\Psi_{1}(x(t))\right).
\end{split}\end{equation}
Therefore, adding \eqref{eq:B11-main} to \eqref{eq:B12-main}, we arrive at 
\begin{equation}
\begin{split}
&2\lambda(t)\beta(t)\inner{\opB_{1}(x(t)),x(t)+\dot{x}(t)-u}\\
&\geq \lambda(t)\beta(t)\left(\frac{2}{(1+c_{1})L_{\Psi_{1}}}-\frac{1+c_{0}}{1+c_{1}}\lambda(t)\beta(t)\right)\norm{\opB_{1}(x(t))}^{2}\\
&+\frac{2c_{1}\lambda(t)\beta(t)}{1+c_{1}}\Psi_{1}(x(t)+\dot{x}(t))-\left(\frac{1}{(1+c_{1})(1+c_{0})}+\frac{c_{1}L_{\Psi_{1}}\lambda(t)\beta(t)}{1+c_{1}}\right)\norm{\dot{x}(t)}^{2}
\end{split}
\end{equation}
Plugging this into \eqref{eq:Lyap1}, we can continue this thread as 
\begin{align*}
\frac{\dif}{\dif t}\norm{x(t)-u}^{2}&\leq 2\lambda(t)\inner{a+\opD(x(t))+\eps(t)x(t),u-\dot{x}(t)-x(t)}\\
&-2\lambda(t)\beta(t)\Psi_{2}(x(t)+\dot{x}(t))-\frac{2c_{1}\lambda(t)\beta(t)}{1+c_{1}}\Psi_{1}(x(t)+\dot{x}(t))\\
&-\lambda(t)\beta(t)\left(\frac{2}{(1+c_{1})L_{\Psi_{1}}}-\frac{1+c_{0}}{1+c_{1}}\lambda(t)\beta(t)\right)\norm{\opB_{1}(x(t))}^{2}\\
&+\left(\frac{1}{(1+c_{1})(1+c_{0})}+\frac{c_{1}L_{\Psi_{1}}\lambda(t)\beta(t)}{1+c_{1}}-2\right)\norm{\dot{x}(t)}^{2}
\end{align*}
We have 
$$
2\lambda(t)\beta(t)\Psi_{2}(x(t)+\dot{x}(t))\geq \frac{2c_{1}\lambda(t)\beta(t)}{1+c_{1}}\Psi_{2}(x(t)+\dot{x}(t)). 
$$
Collecting terms, we therefore arrive at the expression
\begin{align*}
&\frac{\dif}{\dif t}\norm{x(t)-u}^{2}+\lambda(t)\beta(t)\left(\frac{2}{(1+c_{1})L_{\Psi_{1}}}-\frac{1+c_{0}}{1+c_{1}}\lambda(t)\beta(t)\right)\norm{\opB_{1}(x(t))}^{2}\\
&+\left(2-\frac{1}{(1+c_{1})(1+c_{0})}-\frac{c_{1}L_{\Psi_{1}}\lambda(t)\beta(t)}{1+c_{1}}\right)\norm{\dot{x}(t)}^{2}\\
&\leq 2\lambda(t)\inner{a+\opD(x(t))+\eps(t)x(t),u-x(t)-\dot{x}(t)}\\
&-\frac{2c_{1}\lambda(t)\beta(t)}{1+c_{1}}(\Psi_{1}+\Psi_{2})(x(t)+\dot{x}(t)).
\end{align*}

For ease of notation, let us set $\opD_{\eps(t)}\eqdef\opD+\eps(t)\Id_{\scrH}$. We next observe that 
\begin{align*}
&2\lambda(t)\inner{a+\opD_{\eps(t)}(x(t)),u-x(t)-\dot{x}(t)}\\
&=2\lambda(t)\inner{a+\opD_{\eps(t)}(x(t)),u-x(t)}+2\lambda(t)\inner{a+\opD_{\eps(t)}(x(t)),-\dot{x}(t)}\\
&= 2\lambda(t)\inner{a+\opD_{\eps(t)}(u),u-x(t)}+2\lambda(t)\inner{\opD_{\eps(t)}(x(t))-\opD_{\eps(t)}(u),u-x(t)}\\
&\quad+2\lambda(t)\inner{a+\opD_{\eps(t)}(x(t)),-\dot{x}(t)}\\
&\leq -2\lambda(t)\eps(t)\norm{x(t)-u}^{2}+2\lambda(t)\inner{a+\opD_{\eps(t)}(u),u-x(t)}\\
&\quad+2\lambda(t)\inner{a+\opD_{\eps(t)}(x(t)),-\dot{x}(t)}\\
&\leq-2\lambda(t)\eps(t)\norm{x(t)-u}^{2}+2\lambda(t)\inner{a+\opD_{\eps(t)}(u),u-x(t)}\\
&\quad +\frac{c_2}{2}\norm{\dot{x}}^2+\frac{2\lambda^2(t)}{c_2}\norm{a+\opD_{\eps(t)}(x(t))}^2\\
&\leq-2\lambda(t)\eps(t)\norm{x(t)-u}^{2}+2\lambda(t)\inner{a+\opD_{\eps(t)}(u),u-x(t)}\\
&\quad +\frac{c_2}{2}\norm{\dot{x}}^2+\frac{4\lambda^2(t)}{c_2}\norm{a+\opD_{\eps(t)}(u)}^2++\frac{4\lambda^2(t)}{c_2}\norm{\opD_{\eps(t)}(x(t))-\opD_{\eps(t)}(u)}^2\\
&\leq-2\lambda(t)\eps(t)\norm{x(t)-u}^{2}+2\lambda(t)\inner{a+\opD_{\eps(t)}(u),u-x(t)}\\
&\quad +\frac{c_2}{2}\norm{\dot{x}}^2+\frac{4\lambda^2(t)}{c_2}\norm{a+\opD_{\eps(t)}(u)}^2+\frac{4\lambda^2(t)}{c_2}\bigg(\frac{2}{\eta^2}+2\eps(t)^2\bigg)\norm{x(t)-u}^2
\end{align*}
from the definition, we have $a+\opD(u)=w-\xi$, which means
\begin{align*}
&2\lambda(t)\inner{a+\opD_{\eps(t)}(x(t)),u-x(t)-\dot{x}(t)}\\
&\leq\bigg(\frac{8\lambda^2(t)}{c_2\eta^2}+\frac{8\lambda^2(t)\eps^2(t)}{c_2}-2\lambda(t)\eps(t)\bigg)\norm{x(t)-u}^{2}+2\lambda(t)\inner{a+\opD_{\eps(t)}(u),u-x(t)}\\
&\quad +\frac{c_2}{2}\norm{\dot{x}}^2+\frac{4\lambda^2(t)}{c_2}\norm{a+\opD_{\eps(t)}(u)}^2\\
&\leq\bigg(\frac{8\lambda^2(t)}{c_2\eta^2}+\frac{8\lambda^2(t)\eps^2(t)}{c_2}-2\lambda(t)\eps(t)\bigg)\norm{x(t)-u}^{2}+2\lambda(t)\inner{w-\xi,u-x(t)}\\
&\quad +2\lambda(t)\eps(t)\inner{u,u-x(t)}+\frac{c_2}{2}\norm{\dot{x}(t)}^2+\frac{4\lambda^2(t)}{c_2}\norm{a+\opD_{\eps(t)}(u)}^2
\end{align*}
By Young's inequality,
\begin{align*}
2\lambda(t)\eps(t)\inner{u,u-x(t)}&=2\inner{\eps(t)u,\lambda(t)(u-x(t))}\leq c_3\eps^2(t)\norm{u}^2+\frac{1}{c_3}\lambda^2(t)\norm{u-x(t)}^2
\end{align*}
hence,
\begin{align*}
&2\lambda(t)\inner{a+\opD_{\eps(t)}(x(t)),u-x(t)-\dot{x}(t)}\\
&\leq\bigg(\frac{8\lambda^2(t)}{c_2\eta^2}+\frac{8\lambda^2(t)\eps^2(t)}{c_2}+\frac{1}{c_3}\lambda^2(t)-2\lambda(t)\eps(t)\bigg)\norm{x(t)-u}^{2}\\
&\quad+2\lambda(t)\inner{w-\xi,u-x(t)}+c_3\eps^2(t)\norm{u}^2+\frac{c_2}{2}\norm{\dot{x}(t)}^2 +\frac{4\lambda^2(t)}{c_2}\norm{a+\opD_{\eps(t)}(u)}^2.
\end{align*}
Plugging this bound into the above display, we can continue with
\begin{align*}
&\frac{\dif}{\dif t}\norm{x(t)-u}^{2}+\lambda(t)\beta(t)\left(\frac{2}{(1+c_{1})L_{\Psi_{1}}}-\frac{1+c_{0}}{1+c_{1}}\lambda(t)\beta(t)\right)\norm{\opB_{1}(x(t))}^{2}\\
&+\left(2-\frac{1}{(1+c)(1+c_{0})}-\frac{c_{1}L_{\Psi_{1}}\lambda(t)\beta(t)}{1+c_{1}}\right)\norm{\dot{x}(t)}^{2}\\
&\leq 2\lambda(t)\inner{w,u-x(t)}-\frac{2c_{1}\lambda(t)\beta(t)}{1+c_{1}}(\Psi_{1}+\Psi_{2})(x(t)+\dot{x}(t))\\
&\quad+\bigg(\frac{8\lambda^2(t)}{c_2\eta^2}+\frac{8\lambda^2(t)\eps^2(t)}{c_2}+\frac{1}{c_3}\lambda^2(t)-2\lambda(t)\eps(t)\bigg)\norm{x(t)-u}^{2}\\
&\quad+c_3\eps^2(t)\norm{u}^2+\frac{c_2}{2}\norm{\dot{x}(t)}^2 +\frac{4\lambda^2(t)}{c_2}\norm{a+\opD_{\eps(t)}(u)}^2\\
&\quad-2\lambda(t)\inner{\xi,u-x(t)-\dot{x}(t)}-2\lambda(t)\inner{\xi,\dot{x}(t)}.
\end{align*}
Applying Young's inequality again to the last term, we obtain
\[
2\lambda(t)\inner{\xi,-\dot{x}(t)}\leq \frac{c_{2}}{2}\norm{\dot{x}(t)}^{2}+\frac{2\lambda^{2}(t)}{c_{2}}\norm{\xi}^{2}.
\]
Consequently, 
\begin{align*}
    &\frac{\dif}{\dif t}\norm{x(t)-u}^{2}+\lambda(t)\beta(t)\norm{\opB_{1}(x(t))}^{2}\left(\frac{2}{(1+c_{1})L_{\Psi_{1}}}-\frac{1+c_{0}}{1+c_{1}}\lambda(t)\beta(t)\right)\\
&+\left(2-\frac{1}{(1+c_{1})(1+c_{0})}-\frac{c_{1}L_{\Psi_{1}}\lambda(t)\beta(t)}{1+c_{1}}-c_2\right)\norm{\dot{x}(t)}^{2}\\
&\leq 2\lambda(t)\inner{w,u-x(t)}-\frac{2c_{1}\lambda(t)\beta(t)}{1+c_{1}}(\Psi_{1}+\Psi_{2})(x(t)+\dot{x}(t))\\
&\quad+\bigg(\frac{8\lambda^2(t)}{c_2\eta^2}+\frac{8\lambda^2(t)\eps^2(t)}{c_2}+\frac{1}{c_3}\lambda^2(t)-2\lambda(t)\eps(t)\bigg)\norm{x(t)-u}^{2}\\
&\quad+c_3\eps^2(t)\norm{u}^2+\frac{4\lambda^2(t)}{c_2}\norm{a+\opD_{\eps(t)}(u)}^2+\frac{2\lambda^{2}(t)}{c_{2}}\norm{\xi}^{2}-2\lambda(t)\inner{\xi,u-x(t)-\dot{x}(t)}
\end{align*}
Let us set $\delta:=\limsup_{t\to\infty}L_{\Psi_{1}}\lambda(t)\beta(t)<2$. Then, we obtain
\begin{align*}
&2-\frac{1}{(1+c_{1})(1+c_{0})}-\frac{c_{1}L_{\Psi_{1}}\lambda(t)\beta(t)}{1+c_{1}}-c_{2}\geq 2-\frac{1}{(1+c_{1})(1+c_{0})}-\frac{c_{1}\delta}{1+c_{1}}-c_{2}.
\end{align*}
Letting $c_{1}\to 0^{+}$, we obtain 
$$
2-\frac{1}{1+c_{0}}-c_{2}>0.
$$
For $(c_{0},c_{2})>0$ sufficiently small this can be achieved. Henceforth, by continuity, there exists a set of parameters $(c_{0},c_{1},c_{2})$ sufficiently small so that 
\[
b\eqdef 2-\frac{1}{(1+c_{1})(1+c_{0})}-\frac{c_{1}\delta}{1+c_{1}}-c_{2}>0,
\]
meanwhile
\[
e\eqdef \frac{2}{(1+c_{1})L_{\Psi_{1}}}-\frac{1+c_{0}}{(1+c_{1})L_{\Psi_{1}}}\delta=\frac{2-(1+c_{0})\delta}{(1+c_{1})L_{\Psi_{1}}}>0.
\]

Let us further call 
\[
    d\eqdef \frac{2c_{1}}{1+c_{1}},\quad \text{and} \quad\delta(t)\eqdef -\frac{8\lambda^2(t)}{c_2\eta^2}-\frac{8\lambda^2(t)\eps^2(t)}{c_2}-\frac{1}{c_3}\lambda^2(t)+2\lambda(t)\eps(t),
\]
and 
to arrive at the expression 
\begin{align*}
     &\frac{\dif}{\dif t}\norm{x(t)-u}^{2}+e\lambda(t)\beta(t)\norm{\opB_{1}(x(t))}^{2}+b\norm{\dot{x}(t)}^{2}+\frac{d}{2}\lambda(t)\beta(t)(\Psi_{1}+\Psi_{2})(x(t)+\dot{x}(t))\\
&\leq -\frac{d}{2}\lambda(t)\beta(t)(\Psi_{1}+\Psi_{2})(x(t)+\dot{x}(t))\\
&-\frac{d}{2}\lambda(t)\beta(t)\inner{\frac{4\xi}{d\beta(t)},u}+\frac{d}{2}\lambda(t)\beta(t)\inner{\frac{4\xi}{d\beta(t)},\dot{x}(t)+x(t)}+\frac{2\lambda^{2}(t)}{c_{2}}\norm{\xi}^{2}\\ 
&\quad+2\lambda(t)\inner{w,u-x(t)}-\delta(t)\norm{x(t)-u}^{2}+c_3\eps^2(t)\norm{u}^2+\frac{4\lambda^2(t)}{c_2}\norm{a+\opD_{\eps(t)}(u)}^2
\end{align*}

Since $\sigma_{\scrC}(\frac{4\xi}{d\beta(t)})=\inner{\frac{4\xi}{d\beta(t)},u}$ and 
$$
\inner{\frac{4\xi}{d\beta(t)},\dot{x}+x(t)}-(\Psi_{1}+\Psi_{2})(x(t)+\dot{x}(t))\leq(\Psi_{1}+\Psi_{2})^{\ast}(\frac{4\xi}{d\beta(t)}),
$$
we can continue with 
\begin{align*}
     &\frac{\dif}{\dif t}\norm{x(t)-u}^{2}+e\lambda(t)\beta(t)\norm{\opB_{1}(x(t))}^{2}+b\norm{\dot{x}(t)}^{2}+\frac{d}{2}\lambda(t)\beta(t)(\Psi_{1}+\Psi_{2})(x(t)+\dot{x}(t))\\
&\leq\frac{d}{2}\lambda(t)\beta(t)\left[(\Psi_{1}+\Psi_{2})^{\ast}(\frac{4\xi}{d\beta(t)})-\sigma_{\scrC}(\frac{4\xi}{d\beta(t)})\right]+\frac{2\lambda^{2}(t)}{c_{2}}\norm{\xi}^{2}\\ 
&\quad+2\lambda(t)\inner{w,u-x(t)}-\delta(t)\norm{x(t)-u}^{2}+c_3\eps^2(t)\norm{u}^2+\frac{4\lambda^2(t)}{c_2}\norm{a+\opD_{\eps(t)}(u)}^2
\end{align*}

\textbf{Step 2: Asymptotic estimates.}
Now, consider the case where $w=0$, i.e. where $u\in\zer(\opA+\opD+\NC_{\scrC})$. 
Our assumptions on $\eps$ and $\lambda$ imply that there exists $T_{0}>0$ sufficiently large so that $\delta(t)>0$ for all $t\geq T_{0}$. Fix such a $T_{0}$, so that for all $t\geq T_{0}$, we can further simplify the above expression to 
\begin{align}
     &\frac{\dif}{\dif t}\norm{x(t)-u}^{2}+e\lambda(t)\beta(t)\norm{\opB_{1}(x(t))}^{2}+b\norm{\dot{x}(t)}^{2}+\frac{d}{2}\lambda(t)\beta(t)(\Psi_{1}+\Psi_{2})(x(t)+\dot{x}(t)) \nonumber \\
&\leq\frac{d}{2}\lambda(t)\beta(t)\left[(\Psi_{1}+\Psi_{2})^{\ast}(\frac{4\xi}{d\beta(t)})-\sigma_{\scrC}(\frac{4\xi}{d\beta(t)})\right]+\frac{2\lambda^{2}(t)}{c_{2}}\norm{\xi}^{2} \nonumber \\ 
&\quad+c_3\eps^2(t)\norm{u}^2+\frac{4\lambda^2(t)}{c_2}\norm{a+\opD_{\eps(t)}(u)}^2 \nonumber \\
&\leq \frac{d}{2}\lambda(t)\beta(t)\left[(\Psi_{1}+\Psi_{2})^{\ast}(\frac{4\xi}{d\beta(t)})-\sigma_{\scrC}(\frac{4\xi}{d\beta(t)})\right]+\frac{2\lambda^{2}(t)}{c_{2}}\norm{\xi}^{2} \nonumber \\ 
&\quad+(c_3+8\lambda^{2}(t)/c_{2})\eps^2(t)\norm{u}^2+\frac{8\lambda^2(t)}{c_2}\norm{a+\opD(u)}^2. \label{E:G}
\end{align}
Integrating the penultimate display from $T_{0}$ to $T_{1}$ with $T_{1}\geq T_{0}$, we obtain 
\begin{align*}
    &\norm{x(T_{1})-u}^{2}-\norm{x(T_{0})-u}^{2}+\int_{T_{0}}^{T_{1}}e\lambda(t)\beta(t)\norm{\opB_{1}(x(t))}^{2}\dif t\\
    &+\int_{T_{0}}^{T_{1}}b\norm{\dot{x}(t)}^{2}\dif t+\int_{T_{0}}^{T_{1}}\frac{d}{2}\lambda(t)\beta(t)(\Psi_{1}+\Psi_{2})(x(t)+\dot{x}(t))\dif t\leq G,
\end{align*}
where $G$ is the integral of the right-hand side of \eqref{E:G} on $\R_{>0}$, which is finite in view of our assumptions. Using Lemma \ref{lem:Abbas}, we see that $\lim_{t\to\infty}\norm{x(t)-u}$ exists, as well as 
\begin{align*}
&\int_{T}^{\infty} \lambda(t)\beta(t)\norm{\opB_{1}(x(t))}^{2}\dif t<\infty,\\
&\int_{T}^{\infty}\lambda(t)\beta(t)(\Psi_{1}+\Psi_{2})(x(t)+\dot{x}(t))\dif t<\infty\text{ and } \\
&\int_{T}^{\infty}\norm{\dot{x}(t)}^{2}\dif t<\infty. 
\end{align*}
Assuming that $\liminf_{t\to\infty}\lambda(t)\beta(t)>0$, we thus see that 
$$
\lim_{t\to\infty}\norm{\opB_{1}(x(t))}=0,\text{ and }\lim_{t\to\infty}(\Psi_{1}+\Psi_{2})(x(t)+\dot{x}(t))=0. 
$$

\textbf{Step 3: Weak convergence of the ergodic trajectory.}
We claim that every weak limit point of a suitably defined ergodic trajectory must belong to $\zer(\opA+\opD+\NC_{\scrC})$. To derive this conclusion, we fix $T_{0}$, and define the ergodic trajectory
$$
\bar{x}(T)\eqdef \frac{\int_{T_{0}}^{T}\lambda(s)x(s)\dif s}{\int_{T_{0}}^{T}\lambda(s)\dif s} 
$$
as a function of the upper boundary of the integral. 
\revise{Let $\bar{x}_{\infty}$ be a weak accumulation point of $\bar{x}(\cdot)$. Then there exists a sequence $T_n\uparrow\infty$ such that $\bar{x}(T_n)\rightharpoonup\bar{x}_{\infty}$.} We claim that $\bar{x}_{\infty}\in\zer(\opA+\opD+\NC_{\scrC})$. To show this, we integrate again for a given pair $(u,w)\in\gr(\opA+\opD+\NC_{\scrC})$, so that 
\begin{align*}
        \norm{x(T_{1})-u}^{2}-&\norm{x(T_{0})-u}^{2}+\int_{T_{0}}^{T_{1}}e\lambda(t)\beta(t)\norm{\opB_{1}(x(t))}^{2}\dif t\\
    &+\int_{T_{0}}^{T_{1}}b\norm{\dot{x}(t)}^{2}\dif t+\int_{T_{0}}^{T_{1}}\frac{d}{2}\lambda(t)\beta(t)(\Psi_{1}+\Psi_{2})(x(t)+\dot{x}(t))\dif t\\
    &\leq G+\int_{T_{0}}^{T_{1}}2\lambda(t)\inner{w,u-x(t)}\dif t.
\end{align*}
Dividing both sides by $\Lambda(T_{0},T_{1}):= \int_{T_{0}}^{T_{1}}\lambda(t)\dif t$, we arrive at 
\begin{align*}
    -\frac{\norm{x(T_{0})-u}^{2}}{2\Lambda(T_{0},T_{1})}\leq \frac{G}{2\Lambda(T_{0},T_{1})}+\inner{w,u-\frac{\int_{T_{0}}^{T_{1}}\lambda(t)x(t)\dif t}{\Lambda(T_{0},T_{1})}}. 
\end{align*}
\revise{Setting $T_1=T_n$ and letting $n\to\infty$, we obtain}
\revise{
$$
0\leq \inner{w,u-\bar{x}_{\infty}}.
$$
}
Fact \ref{fact:angle} allows us to conclude  $\bar{x}_{\infty}\in\zer(\opA+\opD+\NC_{\scrC})$ and so every weak accumulation point of the trajectory \revise{$\bar x(\cdot)$} must lie in this set. Since $\lim_{t\to\infty}\norm{x(t)-u}$ exists for every $u\in\zer(\opA+\opD+\NC_{\scrC})$,
\revise{the argument of \cite[Lemma 2.3]{Attouch:2010aa}, adapted to the $\lambda$-weighted ergodic trajectory, yields that $\bar{x}(T)$ converges weakly, as $T\to\infty$, to a point in $\zer(\opA+\opD+\NC_{\scrC})$, as claimed.}
\end{appendices}

\bibliography{PenaltyDynamics}

\end{document}